\documentclass[11pt]{amsart}
\usepackage[leqno]{amsmath}
\usepackage{amsthm}
\usepackage{amssymb}
\usepackage{stmaryrd}
\usepackage{mathrsfs}
\usepackage{eucal}
\usepackage{color}
\usepackage{url}
\usepackage[all]{xy}

\newtheorem{theorem}{Theorem}[subsection]
\newtheorem{lemma}[theorem]{Lemma}

\newtheorem{proposition}[theorem]{Proposition}

\theoremstyle{definition}

\theoremstyle{remark}
\newtheorem{remark}[theorem]{Remark}

\def\Z{\mathbf{Z}}

\def\Q{\mathbf{Q}}
\def\F{\mathbf{F}}
\def\P{\mathbf{P}}

\def\A{\mathbf{A}}
\def\T{\mathbf{T}}

\let\ms\mathscr
\let\mf\mathfrak
\let\mc\mathcal

\let\wt\widetilde
\let\ol\overline
\let\ul\underline

\let\lbb\llbracket
\let\rbb\rrbracket
\let\wh\widehat

\DeclareMathOperator{\Hom}{Hom}
\DeclareMathOperator{\Tor}{Tor}

\DeclareMathOperator{\Sym}{Sym}

\DeclareMathOperator{\Spec}{Spec}
\DeclareMathOperator{\MaxSpec}{MaxSpec}

\DeclareMathOperator{\tr}{tr}

\DeclareMathOperator{\im}{im}
\DeclareMathOperator{\codim}{codim}

\DeclareMathOperator{\Gal}{Gal}
\DeclareMathOperator{\val}{val}
\DeclareMathOperator{\uHom}{\underline{Hom}}

\newcommand{\GL}{\mathrm{GL}}
\newcommand{\SL}{\mathrm{SL}}
\newcommand{\lw}[1]{{\textstyle \bigwedge}{}^{#1}\,}
\newcommand{\univ}{\mathrm{univ}}

\newcommand{\un}{\mathrm{un}}
\newcommand{\red}{\mathrm{red}}
\newcommand{\ord}{\mathrm{ord}}
\newcommand{\loc}{\mathrm{loc}}

\def\mat#1#2#3#4{\left( \begin{array}{cc} #1 & #2 \\ #3 & #4 \end{array} \right)}

\title{Singularities of ordinary deformation rings}

\author{Andrew Snowden}
\thanks{The author was partially supported by NSF fellowship DMS-0902661.}

\date{November 15, 2011}

\begin{document}

\begin{abstract}
Let $R^{\univ}$ be the universal deformation ring of a residual representation of a local Galois group.  Kisin
showed that many loci in $\MaxSpec(R^{\univ}[1/p])$ of interest are Zariski closed, and gave a way to study the
generic fiber of the corresponding quotient of $R^{\univ}$.  However, his method gives little information
about the quotient ring before inverting $p$.  We give a method for studying this quotient in certain cases, and
carry it out in the simplest non-trivial case.  Precisely, suppose that $V_0$ is the trivial two dimensional
representation and let $R$ be the unique $\Z_p$-flat and reduced quotient of $R^{\univ}$ such that $\MaxSpec(R[1/p])$
consists of ordinary representations with Hodge--Tate weights 0 and 1.  We describe the functor of points of
(a slightly modified version of) $R$ and show that the irreducile components of $\Spec(R)$ are normal,
Cohen--Macaulay and not Gorenstein.  This has two well-known applications to global deformation rings:
first, a global deformation ring employing these local conditions is torsion-free; and second, Kisin's $R[1/p]=\T[1/p]$
theorem can be upgraded in this setting to an $R=\T$ theorem.
\end{abstract}

\maketitle
\tableofcontents

\section{Introduction}

Let $F$ be a finite extension of $\Q_p$, let $k$ be a finite field and let $V_0$ be a finite dimensional $k$-vector
space carrying a continuous representation of the absolute Galois group $G_F$.  Let $\mc{O}$ be a finite totally
ramified extension of $W(k)$.  There
is then a universal ring $R^{\univ}$ parameterizing (framed) deformations of $V_0$ to $\mc{O}$-algebras.  There are
many loci in $\MaxSpec(R^{\univ}[1/p])$ which are of interest:  for example, one has the locus of crystalline
representations with given Hodge--Tate weights.

For reasons that arose in the study of modularity lifting, one would like to be able to show that certain loci in
$\MaxSpec(R^{\univ}[1/p])$ are Zariski closed, and then study the corresponding reduced quotient of $R^{\univ}[1/p]$.
In favorable cases, one can impose evident deformation conditions and show that the resulting problem is representable
by a quotient of $R^{\univ}$ whose generic fiber has for its maximal ideals the locus in question; one can then study
this quotient of $R^{\univ}$ through its moduli-theoretic description.  However, this approach is not always
viable.

Kisin was the first to obtain general results in this direction.  To explain his method we must introduce some
notation.  Let $X$ be a subset of $\MaxSpec(R^{\univ}[1/p])$.  When Kisin's method is applicable, it produces a
projective morphism $\Theta : Z \to \Spec(R^{\univ})$ of schemes such that $\Theta[1/p]$ is a closed immersion
inducing a bijection between the closed points of $Z[1/p]$ and $X$.  Furthermore, the formal completion
of $Z$ along the special fiber of $\Theta$ has an easily described functor of points.  Let $\Spec(R)$ be the
scheme-theoretic image of $\Theta$.  Then, with
some assumptions on $Z$ which usually hold, $R$ is the unique $\mc{O}$-flat reduced quotient of $R^{\univ}$ with
$\MaxSpec(R[1/p])=X$.  Since $\Theta$ induces an isomorphism $Z[1/p] \to \Spec(R[1/p])$, one can study $R[1/p]$
using the moduli problem solved by the completion of $Z$.

Although the above method is well-suited to proving the existence of the ring $R$ and studying its generic fiber,
it does not yield much information about $R$ itself.  For instance, $R$ is defined as a scheme-theoretic image and so
it is not at all clear how to describe its functor of points.

In this paper, we give a method for analyzing $R$ itself, in certain cases.  We carry this method out in the simplest
non-trivial case,
where $V_0$ is the trivial two dimensional representation and $X$ is a certain set of ordinary representations.
Note that this is not one of the ``favorable cases'' mentioned in the second paragraph, since $V_0$ is trivial.
We are able to show that $R$ (or closely related rings) is normal, Cohen--Macaulay and not Gorenstein.  The
Cohen--Macaulay property has well-known
global consequences: it implies that global deformation rings using this local condition are torsion-free, and can be
used to improve Kisin's $R[1/p]=\T[1/p]$ theorem to an $R=\T$ theorem.  We believe that the method we use here
should work for other two dimensional ordinary deformations, and may be feasible for higher dimensional ordinary
representations.  It is not clear to us if it will be useful in other cases, however.

\subsection{Outline of the method in general}
\label{ss:out}

Let $X$ be an equidimensional Zariski closed subset of $\MaxSpec(R^{\univ}[1/p])$ and let $R$ be the unique
$\mc{O}$-flat reduced quotient of $R^{\univ}$ such that $\MaxSpec(R[1/p])=X$.  To analyze $R$ we proceed along
the following steps:
\begin{enumerate}
\item Come up with a list of equations which make sense in arbitrary deformations of $V_0$ and which hold on the set
$X$.  Let $R^{\dag}$ be the quotient of $R^{\univ}$ by these equations.
\item Show that the map $R^{\univ} \to R$ factors through $R^{\dag}$, and that the resulting map $R^{\dag}[1/p]_{\red}
\to R[1/p]$ is an isomorphism.  This should not be difficult if one found enough equations in step (a).
One should understand $R[1/p]$ from Kisin's method, and so one should now have some understanding of $R^{\dag}[1/p]$.
\item Show that $R^{\dag} \otimes_{\mc{O}} k$ is reduced, equidimensional of the same dimension as $R^{\dag}[1/p]$ and
has the same number of minimal primes as $R^{\dag}[1/p]$.  As far as we know, there is no reason that this must be
true; however, if it is true, it may be tractable to prove, as $R^{\dag} \otimes_{\mc{O}} k$ represents an explicit
deformation problem on $k$-algebras.
\item Appeal to abstract commutative algebra facts to conclude that $R^{\dag}$ is $\mc{O}$-flat and reduced
(see Propositions~\ref{flat} and~\ref{reduced}).
\item Conclude that $R^{\dag} \to R$ is an isomorphism, as both are $\mc{O}$-flat and reduced and the map
$R^{\dag}[1/p]_{\red} \to R[1/p]$ is an isomorphism.
\end{enumerate}
If the above steps can be carried out then one has a moduli-theoretic description of $R$, namely, the equations used
to define $R^{\dag}$.  This can then be used to study $R$.

\subsection{Outline of the method in a specific case}
\label{ss:out2}

Let us give a more detailed outline of the above method in a specific, somewhat easy, case.  Let $V_0$ be the trivial
two dimensional representation and let $X$ be the set of representations which are conjugate to one of the form
\begin{displaymath}
\mat{\chi}{\ast}{}{1},
\end{displaymath}
where $\chi$ denotes the cyclotomic character.  We assume that $\chi=1 \pmod{p}$, otherwise $X$ is empty.  There
are two obvious families of equations that make sense in any deformation and that hold on $X$, namely:
\begin{itemize}
\item $\tr(g)=\chi(g)+1$ for $g \in G_F$.
\item $(g-1)(g'-1)=(\chi(g)-1)(g'-1)$ for $g, g' \in G_F$.
\end{itemize}
Let $R^{\dag}$ be the quotient of $R^{\univ}$ by these equations.  To be more precise, let $\rho^{\univ}:G_F \to
\GL_2(R^{\univ})$ be the universal framed deformation.  Let $I$ be the ideal of $R^{\univ}$ generated by
the elements
\begin{displaymath}
\tr(\rho^{\univ}(g))-\chi(g)-1,
\end{displaymath}
as $g$ varies over $G_F$, as well as the entries of the matrix
\begin{displaymath}
(\rho^{\univ}(g)-1)(\rho^{\univ}(g')-1)-(\chi(g)-1)(\rho^{\univ}(g')-1),
\end{displaymath}
as $g$ and $g'$ vary over $G_F$.  Then $R^{\dag}=R^{\univ}/I$.  Thus $R^{\dag}$ is described as a quotient of
$R^{\univ}$ by an explicit, though uncountable, set of relations.  Nonetheless, it is clear what the functor of
points of $R^{\dag}$ is.

The only really interesting remaining step in the above procedure is (c), i.e., the analysis of the ring
$R^{\dag} \otimes_{\mc{O}} k$.  To understand this ring, it suffices to understand maps $R^{\dag} \to A$ where
$A$ is a $k$-algebra.  From the equations defining $R^{\dag}$, together with the fact that $\chi=1\pmod{p}$, we
see that a map $R^{\univ} \to A$, corresponding to a deformation $V$, factors through $R^{\dag}$ if and only if
$\tr(g \vert_V)=2$ for all $g \in G_F$ and $(g-1)(g'-1)V=0$ for all $g, g' \in G_F$.  It is easy to see that
the action of $G_F$ on such a deformation $V$ factors through the abelianization of $G_F$.
Using class field theory, it is therefore possible to give a very explicit description of $\Spec(R^{\dag}
\otimes_{\mc{O}} k)$ as a certain space of tuples of matrices.
(In fact, it is (a formal completion of) the space $\mc{A}_{d+2}$ discussed in \S \ref{aspace}, with $d=[F:\Q_p]$.)
This space can be analyzed by techniques of algebraic geometry and seen to be integral, normal and Cohen--Macaulay.

All the fine points in the above method carry through in this case, and so the map $R^{\dag} \to R$ is an isomorphism.
This gives a moduli-theoretic description of $R$:  a map $R \to A$ corresponds to a framed deformation of $V_0$ to $A$
where the obvious equations given above hold.  Furthermore, since we know that $R^{\dag}$ is $\mc{O}$-flat and
$R^{\dag} \otimes_{\mc{O}} k$ is Cohen--Macaulay, we find that $R$ itself is Cohen--Macaulay.  A similar argument can
be used to show that $R$ is normal.

\subsection{Generalizations}

There are two natural directions in which one could attempt to extend the results of this paper.  First, one could
remain in the two dimensional case but work with non-ordinary representations.  This seems like it may be
difficult, as it is probably not easy to determine the equations that hold on $X$.
Second, one could remain in the ordinary case but work with higher dimensional representations.
This seems to be a more tractable avenue, since it is easy to come up with many equations that hold on $X$.
Analyzing the special fiber of $R^{\dag}$ may be difficult, however.

\subsection{Plan of paper}

We begin in \S \ref{s:alg} by establishing some miscellaneous results we will need.
In \S \ref{s:modmat}, we study certain moduli spaces of matrices.  These spaces will show up as the
special fibers of the rings $R^{\dag}$ (as we saw in \S \ref{ss:out2}), and it is crucial to know that they are
reduced and have the appropriate number of irreducible components.  In
\S \ref{s:local} we carry out the analysis of ordinary deformation rings in the local case and prove the main results
of the paper.  Finally, in \S \ref{s:global}, we give the applications to global deformation rings.

\subsection*{Acknowledgements}

I would like to thank Bhargav Bhatt, Brian Conrad and Mark Kisin for useful conversations.  I would also like to
thank Steven Sam for his comments on a draft of this paper.

\section{Preliminary results}
\label{s:alg}

In this section we establish some preliminary results that we will need later on.  These are all likely well-known
and quite possibly in the literature already.  However, we do not know convenient references for many of them, and
so include proofs.

\subsection{Rings associated with vector bundles}

Let $X$ be a geometrically integral scheme proper over $k$, let $\eta$ be a vector bundle on $X$ and let $Z$ be the
total space of the dual of $\eta$.  Put
\begin{displaymath}
R=\Gamma(Z, \mc{O}_Z)=\Gamma(X, \Sym(\eta)).
\end{displaymath}
In this section we prove two results about $R$, the first concerning its singularities and the second its minimal
free resolution.  We will only apply these results when $X=\P^1$, but the proofs are just as easy in the general
case.  The first result is the following (the notation is explained following the proposition):

\begin{proposition}
\label{geo-1}
Assume that $X$ is normal and Cohen--Macaulay, that $\eta$ is ample and generated by its global sections and that
\begin{displaymath}
H^i(X, \Sym(\eta))=H^i(X, \Sym(\eta) \otimes \det(\eta) \otimes \omega_X)=0
\end{displaymath}
for $i>0$.  Then the map $Z \setminus Z_0 \to \Spec(R) \setminus \{0\}$ is an isomorphism and $R$ is Cohen--Macaulay.
Furthermore, if
\begin{displaymath}
\dim{H^0(X, \det(\eta) \otimes \omega_X)}>1
\end{displaymath}
then the local ring of $R$ at 0 is not Gorenstein.
\end{proposition}

Let us explain the notation and terminology used above.  The ring $R$ is naturally graded.  Let $R_+$ be the ideal of
positive degree elements.  Then $R/R_+$ is identified with $k$.  We write 0 for the point in $\Spec(R)$ corresponding
to the ideal $R_+$.  We write $\omega_Y$ for the dualizing complex of a variety $Y$ over $k$, which is a coherent
sheaf when $Y$ is Cohen--Macaulay.
A vector bundle $\eta$ on $X$ is said to be \emph{ample} if the line bundle $\mc{O}(1)$ on $\P(\eta^{\vee})$ is ample.
If $\eta$ is a line bundle then $\P(\eta^{\vee})=X$ and $\mc{O}(1)$ is just $\eta$, and so this notion of ample
corresponds to the usual one.  A direct sum of ample bundles is again ample.  We write $Z_0$ for
the image of the zero section of $\pi:Z \to X$.

Before proving the proposition we give some lemmas.  In these lemmas, we do not assume the hypotheses of the
proposition.

\begin{lemma}
\label{geo-1a}
If $\eta$ is generated by global sections then $Z \to \Spec(R)$ is proper.
\end{lemma}

\begin{proof}
Since $\eta$ is generated by its global sections, it is a quotient of $V^* \otimes \mc{O}_X$ for some vector
space $V$.  Thus $Z$ is a closed subset of $X \times V$.  As the map $X \times V \to V$ is proper, we find that
the map $Z \to V$ is proper.  This map factors as $Z \to \Spec(R) \to V$, and so $Z \to \Spec(R)$ is proper.
\end{proof}

\begin{lemma}
\label{geo-1b}
Assume that $X$ is normal, that $\eta$ is ample and that $f:Z \to \Spec(R)$ is proper.  Then $f$ induces an
isomorphism $Z \setminus Z_0 \to (\Spec{R}) \setminus \{0\}$.
\end{lemma}

\begin{proof}
Since $\eta$ is ample,
it follows that, for $z \in Z(\ol{k})$, we have $f(z)=0$ if and only if $z \in Z_0(\ol{k})$.  Furthermore,
if $z \not\in Z_0(\ol{k})$ then one can recover $\pi(z)$ from $f(z)$; since the restriction of $f$ to each fiber of
$\pi$ is injective on $\ol{k}$-points, it follows that the restriction of $f$ to
$Z \setminus Z_0$ is injective on $\ol{k}$-points.

Let $f'$ denote the map $Z \setminus Z_0 \to (\Spec{R}) \setminus \{0\}$ induced by $f$.  As $f$ is proper, so too is
$f'$. We thus find that the image of $f'$ is a closed subscheme of $(\Spec{R}) \setminus \{0\}$.  Since $R$ is
integral and has dimension at most that of $Z$, and $f'$ is injective on $\ol{k}$-points, it follows that $f'$
is surjective.  It follows that $f'$ is finite (Zariski's main theorem) and birational (as it is
generically flat), which implies (by normality) that $f'$ is an isomorphism.
\end{proof}

\begin{lemma}
\label{geo-1c}
Let $f:\wt{Y} \to Y$ be a proper birational map of schemes over $k$, with $\wt{Y}$ Cohen--Macaulay and
$Y$ affine.  Suppose that $H^i(\wt{Y}, \mc{O}_{\wt{Y}})=H^i(\wt{Y}, \omega_{\wt{Y}})=0$
for $i>0$ and that the natural map $f^*:H^0(Y, \mc{O}_Y) \to H^0(\wt{Y}, \mc{O}_{\wt{Y}})$ is an isomorphism.  Then $Y$
is Cohen--Macaulay and $\omega_Y$ is isomorphic to $f_*(\omega_{\wt{Y}})$.
\end{lemma}

\begin{proof}
Grothendieck duality for $f$ states that
\begin{displaymath}
Rf_* R\uHom_{\wt{Y}}(F, f^!G)=R\uHom_Y(Rf_* F, G)
\end{displaymath}
for $F \in D^b(\wt{Y})$ and $G \in D^b(Y)$.  Take $F=\mc{O}_{\wt{Y}}$ and $G=\omega_Y$.  The hypotheses of the lemma
imply that $Rf_* F=\mc{O}_Y$; furthermore, $f^!G=\omega_{\wt{Y}}$.  We thus find $Rf_*(\omega_{\wt{Y}})=\omega_Y$.  The
hypotheses of the lemma imply that $Rf_*(\omega_{\wt{Y}})=f_*(\omega_{\wt{Y}})$.  Thus $\omega_Y$ is concentrated in
a single degree, and so $Y$ is Cohen--Macaulay.
\end{proof}

We now return to the proof of Proposition~\ref{geo-1}.

\begin{proof}[Proof of Proposition~\ref{geo-1}]
The scheme $Z$ is geometrically integral and Cohen--Macaulay, being a vector bundle over such a scheme.  By
hypothesis, $H^i(Z, \mc{O}_Z)=H^i(X, \Sym(\eta))=0$ for $i>0$.  We have
$\omega_Z=\pi^*(\det(\eta) \otimes \omega_X)$; when $X$ is smooth, this can be seen from taking the determinant of
the exact sequence
\begin{displaymath}
0 \to \pi^*(\Omega^1_{X/k}) \to \Omega^1_{Z/k} \to \Omega^1_{Z/X} \to 0
\end{displaymath}
after using the identification $\Omega^1_{Z/X}=\pi^*(\eta)$.  Thus
\begin{displaymath}
H^i(Z, \omega_Z)=H^i(X, \Sym(\eta) \otimes \det(\eta) \otimes \omega_X)
\end{displaymath}
vanishes for $i>0$, by hypothesis.  From Lemmas~\ref{geo-1a} and~\ref{geo-1b}, we see that
$Z \to \Spec(R)$ is proper and birational.  Lemma~\ref{geo-1c} shows that $R$ is Cohen--Macaulay and
$\omega_R=H^0(Z, \omega_Z)$.  Now, $\omega_R$ is a graded $R$-module whose first graded piece is
$H^0(X, \det(\eta) \otimes \omega_X)$.  This space injects into $\omega_R/R_+ \omega_R$.  Thus if $H^0(X,
\det(\eta) \otimes \omega_X)$ has dimension greater than 1 then the fiber of $\omega_R$ at 0 has dimension greater
than 1 and $R$ is not Gorenstein at 0.
\end{proof}

We now turn towards our second result.  Suppose that $\eta$ is generated by its global sections, and write
\begin{displaymath}
0 \to \xi \to \epsilon \to \eta \to 0
\end{displaymath}
with $\epsilon$ a globally free coherent $\mc{O}_X$-module.  Let
$V=\Hom(\epsilon, \mc{O}_X)$, so that we have canonical identifications $\epsilon=V^* \otimes \mc{O}_X$ and
$\Spec(\Sym(\epsilon))=X \times V$.  Put $S=\Sym(V^*)$.  Note that there is a natural map $S \to R$ respecting the
grading.  We regard $k$ as an $S$-module via the identification $S/S_+=k$.

Our second result is the following proposition.  It is the basis of Weyman's ``geometric method'' for studying
syzygies, as exposited in \cite[Ch.~5]{Weyman}.  We include a proof since it is short.

\begin{proposition}
\label{geo-2}
Assume $H^i(X, \Sym(\eta))=0$ for $i>0$.  Then we have an isomorphism of graded vector spaces
\begin{displaymath}
\Tor^n_S(R, k)=\bigoplus_{i \ge n} H^{i-n}(X, \lw{i}{\xi})[i],
\end{displaymath}
where $[ \cdot ]$ indicates the grading.
\end{proposition}

\begin{proof}
Consider the diagram
\begin{displaymath}
\xymatrix{
X \ar[r]^-{i'} \ar[d]_{p'} & X \times V \ar[d]^p \\
\ast \ar[r]^i & V }
\end{displaymath}
where $\ast=\Spec(k)$ and the horizontal maps are zero sections.  We have isomorphisms (in the derived category of
graded $k$-vector spaces)
\begin{displaymath}
\Tor^{\bullet}_S(R, k)=Li^* Rp_* \mc{O}_Z=Rp'_* L(i')^* \mc{O}_Z.
\end{displaymath}
The first isomorphism comes from the fact that $Rp_* \mc{O}_Z=p_* \mc{O}_Z=R$, while the second
is the base change map, which is an isomorphism in this case (as can be seen by applying the projection formula of
\cite[Prop.~5.6]{Hartshorne}, taking $F=\mc{O}_Z$ and $G$ the structure sheaf of the point).
Now, the Koszul complex gives a resolution of $\Sym(\eta)$ as a $\Sym(\epsilon)$-module:
\begin{displaymath}
\cdots \to \Sym(\epsilon) \otimes \lw{2}{\xi} \to \Sym(\epsilon) \otimes \lw{1}{\xi} \to \Sym(\epsilon) \to
\Sym(\eta) \to 0.
\end{displaymath}
The terms of this resolution are graded $\Sym(\epsilon)$-modules, with $\lw{i}{\xi}$ being in degree $i$.  The
differentials have degree one.  Let $q:X \times V \to X$ be the projection.  We can recast the above
resolution as a quasi-isomorphism of complexes of coherent $\mc{O}_{X \times V}$ modules:
\begin{displaymath}
[ \lw{\bullet}(q^*\xi) ] \to \mc{O}_Z,
\end{displaymath}
As the sheaves in the Koszul complex are locally free $\mc{O}_{X \times V}$-modules, we can calculate $L(i')^*$ by
simply applying $(i')^*$.  After doing so all differentials vanish, since the differentials in the Koszul complex
have degree one.  We thus have a quasi-isomorphism
\begin{displaymath}
[ \lw{\bullet}(\xi) ] \to L(i')^* \mc{O}_Z
\end{displaymath}
where the complex on the left has zero differentials.  Applying $Rp'_*$ yields the formula in the statement of the
proposition.
\end{proof}

\subsection{A flatness criterion}

The main result of this section gives a fiberwise criterion for flatness over a discrete valuation ring.  Before
stating it, let us recall a few definitions.  Let $A$ be a noetherian ring.  We say that $A$ is \emph{catenary} if for
any two
prime ideals $\mf{p} \subset \mf{q}$ of $A$, any two maximal chains of prime ideals beginning at $\mf{p}$ and ending at
$\mf{q}$ have the same length.  This is a mild condition satisfied by most rings one encounters; for instance, every
finitely generated algebra over a complete local noetherian ring is catenary.  The \emph{dimension}
(resp.\ \emph{codimension}) of a prime ideal $\mf{p}$ of $A$ is
the length of the longest chain of primes beginning (resp.\ ending) at $\mf{p}$, or, equivalently, the dimension
of $A/\mf{p}$ (resp.\ $A_{\mf{p}}$).  We say that $A$ is \emph{equidimensional} of dimension $d$ if $\dim(\mf{p})
+\codim(\mf{p})=d$ for all prime ideals $\mf{p}$.  If $A$ is a noetherian catenary local ring then $A$ is
equidimensional if and only if its minimal primes all have the same dimension.  We can now state the proposition:

\begin{proposition}
\label{flat}
Let $A$ be a catenary noetherian local ring and let $\pi$ be an element of the maximal ideal of $A$.  Assume the
following:
\begin{itemize}
\item The ring $A/\pi A$ is reduced.
\item The rings $A[1/\pi]$ and $A/\pi A$ are equidimensional of the same dimension and have the same number of
minimal primes.
\end{itemize}
Then $\pi$ is not a zero-divisor in $A$.
\end{proposition}

\begin{remark}
Let $\mc{O}$ be a discrete valuation ring with uniformizer $\pi$ and let $A$ be a catenary noetherian local
$\mc{O}$-algebra.  The above proposition gives a criterion for $A$ to be flat over $\mc{O}$ in terms of conditions on
the fibers of the map $\Spec(A) \to \Spec(\mc{O})$.
\end{remark}

\begin{remark}
The proof of the proposition will yield the following additional piece of information:  extension gives a bijection
between the minimal primes of $A$ and those of $A[1/\pi]$, and similarly for $A/\pi A$ in place of $A[1/\pi]$.
\end{remark}

We now prove the proposition.  If $\pi$ is nilpotent then the proposition is trivial, so assume this is not
the case.  Let $d$ be the common dimension of $A[1/\pi]$ and $A/\pi A$, and let $n$ be the common number of minimal
primes in $A[1/\pi]$ and $A/\pi A$.  Let $I$ be the ideal of $\pi$-power torsion in $A$.  We must show that $I=0$.  Put
$B=A/I$.  Note that $B$ is still catenary, noetherian and local, and $\pi$ is not a zero-divisor in $B$.

\begin{lemma}
\label{flat-1}
No minimal prime of $B$ contains $\pi$.
\end{lemma}

\begin{proof}
Let $\mf{p}_1, \ldots, \mf{p}_n$ be the minimal primes of $B$, and suppose $\pi$ belongs to $\mf{p}_1$.  We
cannot have $n=1$, as $\pi$ is not nilpotent.  Let $x_i$, for $2 \le i \le n$, be an element of $\mf{p}_i$ which
does not belong to $\mf{p}_1$, and let $x$ be the product of the $x_i$.  Thus $x$ belongs to $\mf{p}_2 \cap \cdots
\cap \mf{p}_n$, but is not nilpotent.  As $\pi x$ is nilpotent, we find that $\pi^k$ kills $x^k$ for some $k$, which
contradicts the fact that $\pi$ is not a zero-divisor of $B$. 
\end{proof}

\begin{lemma}
\label{flat-2}
Any prime of $B$ minimal over $(\pi)$ has codimension 1.
\end{lemma}

\begin{proof}
Krull's principal ideal theorem says the codimension is at most 1, while Lemma~\ref{flat-1} says it cannot be 0.
\end{proof}

\begin{lemma}
\label{flat-3}
Let $R$ be a noetherian local domain and let $\pi$ be a non-zero element of the maximal ideal of $R$ such that
$R[1/\pi]$ is a field.  Then $R$ has exactly two prime ideals:  its maximal ideal and the zero ideal.
\end{lemma}

\begin{proof}
Since $R[1/\pi]$ is a field, given any non-zero $x \in R$, we can find a non-zero $y \in R$ such that $xy=\pi^n$ for
some $n$.  It follows that every non-zero principal ideal, and therefore every non-zero ideal, contains a power of
$\pi$.  Thus $\pi$ belongs to every non-zero prime ideal.  We thus see that the codimension one primes of $R$ are in
bijection with the minimal primes of $R/\pi R$, and are thus finite in number.

Let $\mf{p}_1, \ldots, \mf{p}_n$ be the codimension one primes of $R$.  Every element of the maximal ideal $\mf{m}$ of
$R$ belongs to one of these primes, by Krull's principal ideal theorem.  Thus $\mf{m} \subset \bigcup \mf{p}_i$.
However, this implies $\mf{m} \subset \mf{p}_i$ for some $i$, and so there is only one $\mf{p}_i$ and it is maximal.
This completes the proof.
\end{proof}

\begin{lemma}
\label{flat-4}
The ring $B$ is equidimensional of dimension $d+1$ and contains $n$ minimal primes.
\end{lemma}

\begin{proof}
The primes of $B$ not containing $\pi$ correspond to the primes of $B[1/\pi]=A[1/\pi]$ via extension and contraction.
Since no minimal prime of $B$ contains $\pi$, it follows that the minimal primes of $B$ and $B[1/\pi]$ are in bijection,
and so $B$ has $n$ minimal primes.  Let $\mf{p}_0$ be a minimal prime of $B$, and let $\mf{q}_0$ be its extension to
$B[1/\pi]$.  Let $\mf{q}_0 \subset \cdots \subset \mf{q}_d$ be a maximal chain of prime ideals in $B[1/\pi]$ and
let $\mf{p}_i$ be the contraction of $\mf{q}_i$.  Then $B/\mf{p}_d$ is a local domain in which $\pi$ is non-zero
element of the maximal ideal, and $(B/\mf{p}_d)[1/\pi]=B[1/\pi]/\mf{q}_d$ is a field.  It follows from
Lemma~\ref{flat-3} that $B/\mf{p}_d$ is one dimensional, and so there is no prime between $\mf{p}_d$ and the maximal
ideal $\mf{p}_{d+1}$ of $B$.  Thus
$\mf{p}_0 \subset \cdots \subset \mf{p}_{d+1}$ is a maximal chain of length $d+1$.  As $B$ is catenary, every
maximal chain between $\mf{p}_0$ and $\mf{p}_{d+1}$ has length $d+1$, and so $B$ is equidimensional of dimension $d+1$.
\end{proof}

\begin{lemma}
\label{flat-5}
The ring $B/\pi B$ is equidimensional of dimension $d$.
\end{lemma}

\begin{proof}
Let $\mf{p}$ be a minimal prime of $B/\pi B$ and let $\wt{\mf{p}}$ be its inverse image in $B$.  Then $\wt{\mf{p}}$
has codimension 1 by Lemma~\ref{flat-2}, and since $B$ is equidimensional, dimension $d$.  Of
course, $\wt{\mf{p}}$ and $\mf{p}$ have the same dimension.  This shows that all minimal primes of $B/\pi B$ have
dimension $d$, and as $B/\pi B$ is catenary, noetherian and local, the proposition follows.
\end{proof}

\begin{lemma}
\label{flat-6}
Let $\mf{p}$ be a minimal prime of $B$.  Then there exists a prime $\mf{q}$ of $B$ minimal over $(\pi)$ which
contains $\mf{p}$.
\end{lemma}

\begin{proof}
The ring $B/\mf{p}$ is a local domain of dimension $d+1$ in which $\pi$ is non-zero element of the maximal ideal, and
so $B/((\pi)+\mf{p})$ has dimension $d$.  Let $\mf{q}$ be the inverse image in $B$ of a minimal prime of
$B/((\pi)+\mf{p})$ of dimension $d$.  Then $\mf{q}$ has dimension $d$, and since $B$ is equidimensional, codimension 1.
It follows from Lemma~\ref{flat-2} that $\mf{q}$ is minimal over $(\pi)$.
\end{proof}

\begin{lemma}
\label{flat-7}
The ring $B/\pi B$ is regular in codimension 0, that is, if $\mf{q}$ is a minimal prime then $(B/\pi B)_{\mf{q}}$ is
a field.
\end{lemma}

\begin{proof}
Let $\mf{q}$ be a minimal prime of $B/\pi B$ and let $\mf{p}$ be its inverse image in $A/\pi A$.  Since $A/\pi A$ and
$B/\pi B$ are equidimensional of the same dimension, $\mf{p}$ is a minimal prime of $A/\pi A$.  One readily verifies
that the natural map $(A/\pi A)_{\mf{p}} \to (B/\pi B)_{\mf{q}}$ is a surjection of rings.  As $A/\pi A$
is reduced, the former is a field, and so the latter is as well.
\end{proof}

\begin{lemma}
\label{flat-8}
The ring $B/\pi B$ contains at least $n$ minimal primes.
\end{lemma}

\begin{proof}
Let $\mf{q}$ be a minimal prime over $(\pi)$ in $B$.  We claim that $\mf{q}$ contains exactly one minimal prime of $B$.
Since $(B/\pi B)_{\mf{q}}=B_{\mf{q}}/\pi B_{\mf{q}}$ is a field (Lemma~\ref{flat-7}), we find that $\pi B_{\mf{q}}$ is
the maximal ideal of $B_{\mf{q}}$.  Since the maximal ideal of $B_{\mf{q}}$ is principal, any other prime ideal of
$B_{\mf{q}}$ is the zero ideal, which establishes the claim.  Since every minimal prime of $B$ is contained in at least
one prime minimal over $(\pi)$ (Lemma~\ref{flat-6}), it follows that there are at least $n$ primes of $B$ minimal over
$(\pi)$.
\end{proof}

\begin{lemma}
\label{flat-9}
The map $A/\pi A \to B/\pi B$ is an isomorphism.
\end{lemma}

\begin{proof}
It is a map of equidimensional noetherian rings of the same dimension such that the source is reduced and the target
has at least the number of minimal primes as the source.
\end{proof}

\begin{lemma}
\label{flat-10}
We have $I=0$.
\end{lemma}

\begin{proof}
Since $B$ has no $\pi$-torsion, the kernel of the map $A/\pi A \to B/\pi B$ is $I/\pi I$.  Thus $I/\pi I=0$, and so
$I=0$ by Nakayama's lemma.
\end{proof}

\subsection{Two more results from commutative algebra}

Let $A$ be a ring and let $\pi$ be an element of $A$.  In this section we give a pair of results which allow us to
transfer properties of $A/\pi A$ and $A[1/\pi]$ to $A$.

\begin{proposition}
\label{reduced}
Suppose that $\pi$ is not a zero-divisor, $A$ is $(\pi)$-adically separated and $A/\pi A$ is reduced.  Then $A$ is
reduced.
\end{proposition}

\begin{proof}
Suppose $x_1$ is an element of $A$ such that $x_1^n=0$.  Then $x_1=0$ in $A/\pi A$, and so we can write $x_1=\pi x_2$.
We thus find $\pi^n x_2^n=0$, and so, since $\pi$ is not a zero-divisor, $x_2^n=0$.  Reasoning as before, we can
write $x_2=\pi x_3$.  Continuing in this manner, we find that $x_1$ belongs to $(\pi^n)$ for all $n$, and is thus equal
to 0.  This completes the proof.
\end{proof}

\begin{proposition}
\label{normal}
Suppose that $A$ is a domain, $A[1/\pi]$ is normal and $A/\pi A$ is reduced.  Then $A$ is normal.
\end{proposition}

\begin{proof}
Let $x$ be an element of the fraction field of $A$ which is integral over $A$; we must show that $x$ belongs to $A$.
Since $A[1/\pi]$ is normal, $x$ belongs to $A[1/\pi]$, and so we can write $x=y/\pi^n$ with $y \in A$ and $n \ge 0$
minimal.  Assume $n>0$.  Let $\sum_{i=0}^d a_i x^i$ be a monic equation over $A$ satisfied
by $x$.  We then have $\sum_{i=0}^d a_i y^i (\pi^n)^{d-i}=0$.  We find $y^d=0$ in $A/\pi A$, which shows that $y$
belongs to $(\pi)$, contradicting the minimality of $n$.  Thus $n=0$ and $x$ belongs to $A$.
\end{proof}

\subsection{A result from Galois theory}

Let $F/\Q_p$ be a finite extension of degree $d$ and let $F'/F$ be the maximal $p$-power Galois extension in which the
inertia group is abelian and killed by $p$.  Let $G=\Gal(F'/F)$ and let $U$ be the inertia subgroup of $G$.  We thus
have a short exact sequence
\begin{displaymath}
1 \to U \to G \to \Z_p \to 0,
\end{displaymath}
where $\Z_p$ has for a topological generator the arithmetic Frobenius element $\phi$.  We give $U$ the structure of
an $\F_p \lbb T \rbb$-module by letting $T$ act by $\phi-1$.  The following result determines the structure of $U$:

\begin{proposition}
\label{galstruct}
If $F$ contains the $p$th roots of unity then 
$U \cong \F_p \oplus \F_p \lbb T \rbb^{\oplus d}$; otherwise $U \cong \F_p \lbb T \rbb^{\oplus d}$
\end{proposition}

\begin{proof}
Fix an element $\wt{\phi}$ in $G$ lifting $\phi$.
Let $F_n$ be the unramified extension of $F$ of degree $p^n$, and let $F_n'$ be the maximal abelian extension of $F_n$
of exponent $p$ on which $\wt{\phi}^{p^n}$ acts trivially.  Then $F_n'$ is the fixed field of the normal closure of
$\langle \wt{\phi} \rangle \subset G$ acting on $F'$, and so $\Gal(F_n'/F_n)$ is identified with $U/T^{p^n} U$.  By
class field theory,
the abelianized Galois group of $F_n$ is identified with $(F_n^{\times})^{\wedge}$, with the element $\wt{\phi}^{p^n}$
corresponding to some element $x_n$ of valuation 1.  It follows that we have isomorphism
\begin{displaymath}
\Gal(F_n'/F_n)=(F_n^{\times})^{\wedge}/\langle x_n \rangle \otimes \F_p=U_{F_n} \otimes \F_p,
\end{displaymath}
where $U_{F_n}$ is the unit group of $F_n$.  We thus find that $U/T^{p^n} U$ has dimension $\epsilon+p^n d$ over
$\F_p$, where $\epsilon$ is 1 or 0 according to whether $F$ contains the $p$th roots of unity or not.  Now, since $U$
is the
inverse limit of the $U/T^{p^n} U$ and $U/TU$ is finite dimensional, it follows that $U$ is finitely generated over
$\F_p \lbb T \rbb$.  Appealing to the structure theory of finitely generated $\F_p \lbb T \rbb$-modules and our
formula for the dimension of $U/T^{p^n} U$ gives the stated result.
\end{proof}

\begin{remark}
Suppose that $F$ contains the $p$th roots of unity.  By the above result, $U$ contains a unique $\F_p$-line on which
$\phi$ acts trivially.  Let us now describe this line more explicitly.  As $F_n$ contains the $p$th roots of unity,
there is some element $y_n$ of $F_n^{\times} \otimes \F_p$ such that $F_{n+1}=F_n(y_n^{1/p})$.  This element is
unique up to scaling by elements of $\F_p^{\times}$.  Alternatively, fixing a $p$th root of unity and letting
$(,)$ be the $\F_p$-valued Hilbert symbol on $F_n^{\times} \otimes \F_p$, we can characterize $y_n$ uniquely by
$(y_n, x)=\val(x)$.  The element $y_n$ is a unit, invariant under $\phi$ and satisfies $N(y_{n+1})=y_n$, where
$N:F_{n+1}^{\times} \to F_n^{\times}$ is the norm map.  The sequence
$(y_n)$ defines an element of the inverse limit of the $U_{F_n} \otimes \F_p$ (where the transition maps are the
norm maps).  By the above proof, this inverse limit is $U$.  Since $(y_n)$ is non-zero and $\phi$-invariant, it spans
the unique $\phi$-invariant $\F_p$-line in $U$.
\end{remark}

\section{Moduli spaces of matrices}
\label{s:modmat}

In this section we study three moduli problems related to $2 \times 2$ nilpotent matrices.  Certain local rings of
these spaces will later be identified as the special fiber of certain Galois deformation rings.  Throughout, $k$ is a
fixed field of characteristic not 2.

\subsection{Borel and nilpotent matrix algebras}

Let $\mf{g}=M_2$ be the space of $2 \times 2$ matrices over $k$ and let $\mf{g}^{\circ}$ be the subspace of traceless
matrices.  Since 2 is invertible, the space $\mf{g}$ is the direct sum of $\mf{g}^{\circ}$ and the space of scalar
matrices.  We write $\ul{\mf{g}}$ and $\ul{\mf{g}}^{\circ}$ for the sheaves $\mf{g} \otimes \mc{O}$ and $\mf{g}^{\circ}
\otimes \mc{O}$ on $\P^1$.

Let $T$ be a $k$-scheme.  A \emph{nilpotent subalgebra} of $\mf{g}_T$ is a line subbundle $\mf{u}$ of $\mf{g}_T$
such that $\ker(\mf{u})=\im(\mf{u})$ is a line subbundle of $\mc{O}_T^2$.  One easily sees that sending $\mf{u}$ to
$\ker(\mf{u})$ defines a bijection between nilpotent subalgebras of $\mf{g}_T$ and line subbundles of $\mc{O}_T^2$.
One recovers $\mf{u}$ from a line bundle $\mc{L}$ via the formula $\mf{u}=\uHom(\mc{O}_T^2/\mf{L}, \mc{L})$.  Since
the bundle $\mc{O}(-1)$ on $\P^1$ is the universal line subbundle of $\mc{O}^2$, it follows that the nilpotent
subalgebra $\ul{\mf{u}}=\uHom(\mc{O}^2/\mc{O}(-1), \mc{O}(-1)) \cong \mc{O}(-2)$ of $\ul{\mf{g}}$ on $\P^1$ is the
universal nilpotent subalgebra of $\mf{g}$.

A \emph{Borel subalgebra} of $\mf{g}_T$ is a rank three subbundle $\mf{b}$ of $\mf{g}_T$ for which there
exists a line subbundle $\mc{L}$ of $\mc{O}_T^2$ such that $\mf{b} \mc{L} \subset \mc{L}$.  Again, one finds that
the correspondence between $\mf{b}$ and $\mc{L}$ is bijective, and so there is a universal Borel subalgebra
$\ul{\mf{b}}$ of $\ul{\mf{g}}$ on $\P^1$.  One can make similar definitions in the traceless case, and obtain a
universal
Borel subalgebra $\ul{\mf{b}}^{\circ}$ of $\ul{\mf{g}}^{\circ}$ on $\P^1$.  We have $\ul{\mf{b}}=\ul{\mf{b}}^{\circ}
\oplus \mc{O}$ since 2 is invertible.  As every section of $\ul{\mf{b}}^{\circ}$ induces an endomorphism of
$\mc{O}^2/\mc{O}(-1)$, we get a canonical map $\ul{\mf{b}}^{\circ} \to \mc{O}$, the kernel of which is $\ul{\mf{u}}$.
That is, we have an exact sequence
\begin{displaymath}
0 \to \ul{\mf{u}} \to \ul{\mf{b}}^{\circ} \to \mc{O} \to 0
\end{displaymath}
of sheaves on $\P^1$.  A global section of $\ul{\mf{b}}^{\circ}$ is determined by the endomorphism of $\mc{O}^2$
it induces, and this endomorphism has image in $\mc{O}(-1)$.  Since there are no non-zero maps $\mc{O}^2 \to
\mc{O}(-1)$, we conclude that $\Gamma(\P^1, \ul{\mf{b}}^{\circ})=0$.  This, combined with the above exact sequence,
implies that $\ul{\mf{b}}^{\circ} \cong \mc{O}(-1)^{\oplus 2}$.

\subsection{Strongly nilpotent matrices}

Let $T$ be a $k$-algebra.  We say that a matrix $m$ in $M_2(T)$ is \emph{strongly nilpotent} if its trace and
determinant are both 0.  This implies $m^2=0$, but the converse does not hold in general (it does if $T$ is a domain).
If $\mf{u}$ is a nilpotent subalgebra of $(M_2)_T$ then every element of $\mf{u}$ is strongly nilpotent.

\subsection{The space $\mc{A}$}
\label{aspace}

Let $r \ge 1$ be an integer.  Let $\mc{A}=\mc{A}_r$ be the functor which assigns to a $k$-algebra $T$ the set
$\mc{A}(T)$ of tuples $(m_1, \ldots, m_r)$ where each $m_i \in M_2(T)$ is a strongly nilpotent matrix such that
$m_i m_j=0$ for all $i$ and $j$.  We let $a \in \mc{A}(k)$ denote the tuple $(0, \ldots, 0)$.  The main result of
this section is the following theorem:

\begin{theorem}
\label{thm-a}
The functor $\mc{A}$ is (represented by) a geometrically integral normal Cohen--Macaulay affine scheme of dimension
$r+1$.  For $r>1$ the local ring at $a$ is not Gorenstein.
\end{theorem}

It is clear that $\mc{A}$ is an affine scheme.  In fact, write
\begin{displaymath}
m_i=\mat{a_i}{b_i}{c_i}{-a_i}.
\end{displaymath}
Then $\mc{A}=\Spec(R)$, where $R$ is the quotient of $k[a_i, b_i, c_i]_{1 \le i \le r}$ by the equations
\begin{equation}
\label{a-eq}
a_ia_j=b_i c_j, \qquad a_i b_j=a_j b_i, \qquad a_i c_j=a_j c_i.
\end{equation}
For $i \ne j$, these equations express the identity $m_i m_j=0$.  For $i=j$, the latter two equations are trivial,
while the first expresses that $m_i$ has determinant 0.

Let $\wt{\mc{A}}$ be the functor which attaches to a $k$-algebra $T$ the set of tuples $(\mf{u}; m_1, \ldots, m_r)$
where $\mf{u}$ is a nilpotent subalgebra of $(M_2)_T$ and each $m_i$ belongs to $\mf{u}$.  It is clear that
$\wt{\mc{A}}$ is represented by the total space of the vector bundle $\ul{\mf{u}}^{\oplus r}$ over $\P^1$, and is
thus smooth and geometrically integral of dimension $r+1$.  Let $\wt{R}$ be the ring of global functions on
$\wt{\mc{A}}$.  Then $\wt{R}$ is normal and geometrically integral.

Let $(\mf{u}; m_1, \ldots, m_r)$ be an element of $\wt{\mc{A}}(T)$.  Then each $m_i$ is strongly nilpotent, and
$m_i m_j=0$ for all $i$ and $j$.  Thus $(m_1, \ldots, m_r)$ defines an element of $\mc{A}(T)$.  We therefore have
a map of schemes $\wt{\mc{A}} \to \mc{A}$, and thus a corresponding map of rings $R \to \wt{R}$.

\begin{lemma}
The map $R \to \wt{R}$ is an isomorphism.
\end{lemma}

\begin{proof}
Let $\eta^{\vee}$ be the vector bundle $\ul{\mf{u}}^{\oplus r}$ on $\P^1$ and let $\epsilon^{\vee}$ be the bundle
$(\ul{\mf{g}}^{\circ})^{\oplus r}$.  Then $\eta^{\vee}$ is naturally a subbundle of $\epsilon^{\vee}$; let
$\xi^{\vee}$ be the quotient bundle.  We have isomorphisms $\eta=\mc{O}(2)^{\oplus r}$ and $\xi=\mc{O}(1)^{\oplus 2r}$.
Note that $\wt{R}=\Gamma(\P^1, \Sym(\eta))$.  Let $S$ be the polynomial ring $k[a_i, b_i, c_i]$, which is
identified with $\Gamma(\P^1, \Sym(\epsilon))$.

We now apply Proposition~\ref{geo-2}.  We find
\begin{displaymath}
\wt{R}/S_+ \wt{R}=H^0(\P^1, \mc{O}) \oplus H^1(\P^1, \xi)[1]=k,
\end{displaymath}
and so $S \to \wt{R}$ is surjective (by Nakayama's lemma).  Let $I$ be the kernel.  Then $\Tor^1_S(R, k)$ is
identified with $I/S_+ I$, and so Proposition~\ref{geo-2} gives
\begin{displaymath}
I/S_+ I=H^0(\P^1, \xi)[1] \oplus H^1(\P^1, \lw{2}{\xi})[2].
\end{displaymath}
The $H^0$ vanishes and the $H^1$ has dimension $\binom{r}{2}$.  Thus $I$ is generated in degree 2 and its degree 2
piece has dimension $\binom{r}{2}$.  An elementary argument shows that the $\binom{r}{2}$ quadratic elements of $S$
given in equation \eqref{a-eq} are linearly independent, and so the kernel of $S \to R$ is equal to $I$.  Thus
the map $R \to \wt{R}$ is an isomorphism.
\end{proof}

The theorem now follows immediately from Proposition~\ref{geo-1}.  To be precise, the above lemma shows that $R$
is normal and geometrically integral.  Let $\eta$ as above.  The space $Z$ in Proposition~\ref{geo-1} is just
$\wt{A}$, and the point 0 of $\Spec(R)$ is just $a$.  The proposition shows that $\wt{\mc{A}} \to \mc{A}$ is an
isomorphism away from $a$, that $\wt{R}=R$ is Cohen--Macaulay and that the local ring of $R$ at $a$ is not Gorenstein
when $r>1$.

\begin{remark}
The ring $R$ is isomorphic to the projective coordinate ring of $\P^1 \times \P^r$ with respect to the bundle
$\mc{O}(2, 1)$.  The singularities of such Segre--Veronese rings have been well-studied, and more general results
than Theorem~\ref{thm-a} appear in the literature; see, for instance, \cite[\S 0.4]{BarcanescuManolache}.
\end{remark}

\subsection{The space $\mc{B}$}
\label{bspace}

Let $\mc{B}=\mc{B}_r$ be the functor which assigns to a $k$-algebra $T$ the set $\mc{B}(T)$ of tuples $(\phi,
\alpha; m_1, \ldots, m_r)$ where $\phi$ is an element of $M_2(T)$ of determinant 1, $\alpha$ is an element of $T$
which is a root of the characteristic polynomial of $\phi$ and the $m_i$ are strongly nilpotent matrices in $M_2(T)$
such that $m_i m_j=0$ and $m_i \phi=\alpha m_i$.  Let $b \in \mc{B}(k)$ be the point $(1, 1; 0, \ldots, 0)$.  The
main result of this section is the following:

\begin{theorem}
\label{thm-b}
The functor $\mc{B}$ is (represented by) a geometrically integral normal Cohen--Macaulay affine scheme of dimension
$r+3$.  The local ring at the point $b$ is not Gorenstein.
\end{theorem}

It is clear that $\mc{B}$ is an affine scheme.  In fact, write
\begin{displaymath}
\phi=\mat{\phi_1}{\phi_2}{\phi_3}{\phi_4}
\end{displaymath}
and keep the notation for the $m_i$ from the previous section.  Then $\mc{B}=\Spec(R)$, where $R$ is the quotient of
$k[a_i, b_i, c_i, \phi_j, \alpha]$ (with $1 \le i \le r$ and $1 \le j \le 4$) by the equations \eqref{a-eq}, the
equations
\begin{equation}
\label{b-eq-1}
a_i \phi_1+b_i \phi_3=\alpha a_i, \qquad
a_i \phi_2+b_i \phi_4=\alpha b_i, \qquad
c_i \phi_1-a_i \phi_3=\alpha c_i, \qquad
c_i \phi_2-a_i \phi_4=-\alpha a_i,
\end{equation}
the equation
\begin{equation}
\label{b-eq-2}
\alpha^2-(\phi_1+\phi_4) \alpha+(\phi_1 \phi_4-\phi_2 \phi_3)=0
\end{equation}
and the equation
\begin{equation}
\label{b-eq-3}
\phi_1 \phi_4-\phi_2 \phi_3=1.
\end{equation}
Of course, the equaion \eqref{b-eq-1} express the identity $m_i \phi=\alpha m_i$, while \eqref{b-eq-2} expresses that
$\alpha$ is a root of the characteristic polynomial of $\phi$ and \eqref{b-eq-3} expresses that $\phi$ has determinant
1.

Let $\mc{B}^{\circ}$ be defined like $\mc{B}$ except without any condition on the determinant of $\phi$.  Then
$\mc{B}^{\circ}=\Spec(R^{\circ})$, where $R^{\circ}$ is the quotient of $k[a_i, b_i, c_i, \phi_j, \alpha]$ by the
equations \eqref{a-eq}, \eqref{b-eq-1} and \eqref{b-eq-2}.  Note that these equations are all homogeneous of degree
two, and so $R^{\circ}$ is graded.

Let $\wt{\mc{B}}^{\circ}$ be the functor assigning to a $k$-algebra $T$ the set of tuples
$(\mf{b}; \phi; m_1, \ldots, m_r)$ where $\mf{b}$ is a Borel subalgebra of $\mf{g}_T$, $\phi$ is an element of $\mf{b}$
and the $m_i$ are elements of $\mf{u}$, the nilpotent radical of $\mf{b}$.  It is clear that $\wt{\mc{B}}^{\circ}$
is represented by the total space of the vector bundle $\ul{\mf{b}} \oplus \ul{\mf{u}}^{\oplus r}$ over $\P^1$.  We
write $\wt{R}^{\circ}$ for the ring of global functions on $\wt{B}^{\circ}$.

We let $\wt{\mc{B}}$ be defined like $\wt{\mc{B}}^{\circ}$ but with the condition $\det(\phi)=1$ imposed.  Note
that $\wt{\mc{B}}$ is the fiber product of the universal Borel subgroup of $\SL(2)$ with a vector bundle over $\P^1$,
and is therefore smooth and geometrically integral of dimension $r+3$.  We let $\wt{R}$ be the ring of global
functions on $\wt{\mc{B}}$.

Let $(\mf{b}; \phi; m_1, \ldots, m_r)$ be an element of $\wt{\mc{B}}^{\circ}(T)$.  Let $\alpha$ be the eigenvalue
by which $\phi$ acts on the $[\mf{b}, \mf{b}]$ coinvariants of $T^2$ (e.g., the lower right entry of $\phi$ if $\mf{b}$
is upper triangular).  Then $\alpha$ satisfies the characteristic
polynomial of $\phi$, and $m_i \phi=\alpha m_i$.  It follows that $(\phi, \alpha; m_1, \ldots, m_r)$ is an element of
$\mc{B}^{\circ}$.  We thus have a natural map $\wt{\mc{B}}^{\circ} \to \mc{B}^{\circ}$, and thus an induced map
$R^{\circ} \to \wt{R}^{\circ}$.  Of course, we also have $\wt{\mc{B}} \to \mc{B}$ and $R \to \wt{R}$.

We summarize the above definitions with two commutative diagrams
\begin{displaymath}
\xymatrix{
\wt{\mc{B}} \ar@{^(->}[r] \ar[d] & \wt{\mc{B}}^{\circ} \ar[d] \\
\mc{B} \ar@{^(->}[r] & \mc{B}^{\circ} }
\qquad\qquad
\xymatrix{
\wt{R} & \wt{R}^{\circ} \ar[l] \\
R \ar[u] & R^{\circ} \ar@{->>}[l] \ar[u] }
\end{displaymath}
We now show that the vertical ring maps in the right diagram are isomorphisms.

\begin{lemma}
\label{lem-b1}
The map $R^{\circ} \to \wt{R}^{\circ}$ is an isomorphism.
\end{lemma}

\begin{proof}
Let $\eta^{\vee}$ be the vector bundle $\ul{\mf{b}} \oplus \ul{\mf{u}}^{\oplus r}$ on $\P^1$, which is isomorphic
to $\mc{O} \oplus \mc{O}(-1)^{\oplus (r+2)}$.  The bundle $\eta^{\vee}$ is naturally
a subbundle of the constant bundle $\epsilon^{\vee}=\ul{\mf{g}} \oplus (\ul{\mf{g}}^{\circ})^{\oplus r}$.
The quotient bundle $\xi^{\vee}$ is isomorphic to $\mc{O}(2) \oplus \mc{O}(1)^{\oplus 2r}$.  The ring $\wt{R}^{\circ}$
is $\Gamma(\P^1, \Sym(\eta))$.  Let $S$ be the ring $\Gamma(\P^1, \Sym(\epsilon))$, i.e., the polynomial ring
$k[a_i, b_i, c_i, \phi_i]$.  (Note: $S$ does not contain $\alpha$.)  Then $\wt{R}^{\circ}$ is naturally an
$S$-algebra, as is $R^{\circ}$, and the map $R^{\circ} \to \wt{R}^{\circ}$ is $S$-linear.

It is clear that $R^{\circ}/S_+ R^{\circ}$ is isomorphic to $k \oplus k[1]$, and spanned by the images of 1 and
$\alpha$.  Proposition~\ref{geo-2} shows that
\begin{displaymath}
\wt{R}^{\circ}/S_+ \wt{R}^{\circ}=H^0(\P^1, \mc{O}) \oplus H^1(\P^1, \xi)[1]=k \oplus k[1].
\end{displaymath}
Since $\alpha$ maps to a non-zero element of $\wt{R}^{\circ}$, we see that the map $R^{\circ} \to \wt{R}^{\circ}$
induces an isomorphism after quotienting by $S_+$, and is therefore surjective.

Consider the diagram
\begin{displaymath}
\xymatrix{
0 \ar[r] & M \ar[d] \ar[r] & S \oplus S[1] \ar[r] \ar@{=}[d] & R^{\circ} \ar[d] \ar[r] & 0 \\
0 \ar[r] & N \ar[r] & S \oplus S[1] \ar[r] & \wt{R}^{\circ} \ar[r] & 0 }
\end{displaymath}
The map $S \oplus S[1] \to R^{\circ}$ sends the basis vectors to 1 and $\alpha$.  The map to $\wt{R}^{\circ}$ is
defined similarly.  These maps are surjective by the previous paragraph.  We let $M$ and $N$ denote their kernels,
so that $M \subset N$.  Now, the equations for $\wt{R}^{\circ}$ are all of degree 2, which means that the natural
map $M_2 \to M/S_+M$ is an isomorphism.  On the other hand,
$N/S_+N$ is identified with $\Tor_1^S(\wt{R}^{\circ}, k)$, and so another application of Proposition~\ref{geo-2}
yields
\begin{displaymath}
N/S_+ N=H^0(\P^1, \xi)[1] \oplus H^1(\P^1, \lw{2}{\xi})[2],
\end{displaymath}
and so $N/S_+ N$ is concentrated in degree 2 and of dimension $4r+\binom{2r}{2}$.  We thus have a diagram
\begin{displaymath}
\xymatrix{
M_2 \ar[r] \ar[d] & M/S_+ M \ar[d] \\
N_2 \ar[r] & N/S_+ N }
\end{displaymath}
in which the horizontal maps are isomorphisms.  Obviously, the left vertical map is injective.  An elementary
argument shows the $\binom{2r}{2}$ equations in \eqref{a-eq} and the $4r$ equations in \eqref{b-eq-1} are linearly
independent, and so $M_2$ has dimension $4r+\binom{2r}{2}$.  (Note: the equation \eqref{b-eq-2} does not appear in
$M$.)  We therefore find that $M_2 \to N_2$ is surjective,
and so $M/S_+ M \to N/S_+ N$ is surjective, and so $M \to N$ is surjective, i.e., $M=N$.  This completes the proof.
\end{proof}

\begin{lemma}
\label{lem-b3}
The map $R \to \wt{R}$ is an isomorphism.
\end{lemma}

\begin{proof}
Due to the previous lemma, it is enough to show that $\wt{R}^{\circ} \to \wt{R}$ is surjective with kernel generated
by $\det(\phi)-1$.  Now, the space $\wt{\mc{B}}^{\circ}$ is affine over $\P^1$ and corresponds
to the sheaf of algebras $\Sym(\eta)$.  The space $\wt{\mc{B}}$ is also affine over $\P^1$; let $\mc{R}$ denote the
corresponding sheaf of algebras.  We have a short exact sequence of sheaves on $\P^1$
\begin{displaymath}
0 \to \Sym(\eta) \to \Sym(\eta) \to \mc{R} \to 0
\end{displaymath}
where the first map is given by multiplication by $\det(\phi)-1$, an element of $\Gamma(\P^1, \Sym(\eta))$.  Since
$H^1(\P^1, \Sym(\eta))=0$, upon taking global sections we see that $\wt{R}^{\circ} \to \wt{R}$ is surjective and its
kernel is generated by $\det(\phi)-1$.
\end{proof}

Let $\mc{B}^{\circ}_0(T)$ be the subset of $\mc{B}^{\circ}(T)$ consisting of tuples $(\phi, \alpha; m_1, \ldots, m_r)$
where $m_i=0$ and $\phi=\alpha$ is a scalar matrix.  Define $\wt{\mc{B}}^{\circ}$ similarly.

\begin{lemma}
\label{lem-b2}
The map $\wt{\mc{B}}^{\circ} \setminus \wt{\mc{B}}^{\circ}_0 \to \mc{B}^{\circ} \setminus \mc{B}^{\circ}_0$ is an
isomorphism.  The space $\mc{B}^{\circ}$ is Cohen--Macaulay; moreover, it is smooth away from $\mc{B}^{\circ}_0$ and
not Gorenstein at any point in $\mc{B}^{\circ}_0$.
\end{lemma}

\begin{proof}
Let $\eta$ be as in Lemma~\ref{lem-b1}.  The $\mc{O}$ summand of $\eta$ comes from the scalar matrices in $\mf{b}$.
Let $\eta'$ be the complement of this, so that $\eta=\eta' \oplus \mc{O}$.  Put $A=\Gamma(\P^1, \Sym(\eta'))$.  Then
$\wt{R}^{\circ}=R^{\circ}$ is the polynomial ring in one variable over $A$, and $\mc{B}^{\circ}=\Spec(A)
\times \A^1$, with $\{0 \} \times \A^1$ identified with $\mc{B}^{\circ}_0$.  Proposition~\ref{geo-1}
applies to $\eta'$.  We find that the map from the total space of $\eta'$ to $\Spec(A)$ is an isomorphism away from 0,
that $A$ is Cohen--Macaulay and that the local ring of $A$ at the point 0 is not Gorenstein.  The lemma follows.
\end{proof}

\begin{lemma}
The map $\wt{\mc{B}} \to \mc{B}$ is birational.
\end{lemma}

\begin{proof}
This follows immediately from the previous lemma and the fact that $\mc{B}^{\circ}_0 \cap \mc{B}$ is a proper closed
set in $\mc{B}$ (in fact, it consists of two closed points).
\end{proof}

We now prove the theorem.

\begin{proof}[Proof of Theorem~\ref{thm-b}]
Since $\wt{\mc{B}}$ is normal and geometrically integral, so is $R=\wt{R}$.  As $\wt{\mc{B}} \to \mc{B}$ is birational,
$R$ has dimension $r+3$.
Since $R$ is the quotient of the Cohen--Macaulay ring $R^{\circ}$ by the non-zerodivisor $\det(\phi)-1$, it is
Cohen--Macaulay \cite[Thm.~2.1.3a]{BrunsHerzog}.  Furthermore, since $R^{\circ}$ is not Gorenstein at the point $b$, and
$\det(\phi)-1$ belongs to the maximal ideal at $b$, we see that $R$ is not Gorenstein at $b$
\cite[Prop.~3.1.19b]{BrunsHerzog}.
\end{proof}

\subsection{The space $\mc{C}$}
\label{cspace}

Let $\mc{C}=\mc{C}_r$ be the functor which assigns to a $k$-algebra $T$ the set $\mc{C}(T)$ of tuples
$(\phi, \alpha; m_1, \ldots, m_{r+1})$ where $\phi$ is an element of $M_2(T)$ of determinant 1, $\alpha$ is
an element of $T$ satisfying the characteristic polynomial of $\phi$ and the $m_i$ are strongly nilpotent matrices
in $M_2(T)$ such that the following equations hold:
\begin{displaymath}
m_i m_j=0, \qquad m_i \phi=\alpha m_i, \qquad m_{r+1} \phi=\phi m_{r+1}.
\end{displaymath}
The first two equations are for $1 \le i, j \le r+1$.  The second equation is equivalent to $\phi m_i=\alpha^{-1} m_i$,
as can be seen by taking the adjugate of each side.  In particular, the final equation is equivalent to
$(\alpha^2-1)m_{r+1}=0$.  We let $c \in \mc{C}(k)$ denote the point
$(1, 1; 0, \ldots, 0)$.  The main result of this section is the following theorem:

\begin{theorem}
\label{thm-c}
The functor $\mc{C}$ is (represented by) a geometrically reduced affine scheme which is equidimensional of dimension
$r+3$.  It has three irreducible components:  two isomorphic to $\mc{A}_{r+2}$ (defined by the equations
$\alpha=\pm 1$) and one isomorphic to $\mc{B}_r$ (defined by the equation $m_{r+1}=0$).
\end{theorem}

We require the following simple lemma.

\begin{lemma}
Let $R$ be a ring, let $\mf{p}$ be a prime ideal of $R$ and let $\mf{a}$ be a principal ideal of $R$ not contained in
$\mf{p}$.  Then $\mf{a} \cap \mf{p}=\mf{a} \mf{p}$.
\end{lemma}

\begin{proof}
Clearly, $\mf{a} \mf{p}$ is contained in $\mf{a} \cap \mf{p}$.  We now establish the reverse inclusion.  Let
$\mf{a}=(a)$ and let $x$ be an element of $\mf{a} \cap \mf{p}$.  We can then write $x=ay$ for some $y \in R$.  Since
$ay$ belongs to $\mf{p}$ but $a$ does not belong to $\mf{p}$, we conclude that $y$ belongs to $\mf{p}$.  Thus $x$
belongs to $\mf{a} \mf{p}$.
\end{proof}

We now prove the theorem.

\begin{proof}[Proof of Theorem~\ref{thm-c}]
It is clear that $\mc{C}$ is represented by an affine scheme $\Spec(R)$; we do not write the equations, but keep our
previous notation for elements of the ring $R$.  The locus $\alpha=\pm 1$ in
$\mc{C}$ is isomorphic to $\mc{A}_{r+2}$, the isomorphism taking a $T$-point $(\phi, \alpha; m_1, \ldots, m_{r+1})$ to
$(\phi-\alpha, m_1, \ldots, m_{r+1})$.  The locus $m_{r+1}=0$ in $\mc{C}$ is isomorphic to $\mc{B}_r$ (obviously).  It
follows that $\mf{p}_1=(\alpha-1)$, $\mf{p}_2=(\alpha+1)$ and $\mf{p}_3=(a_{i+1}, b_{i+1}, c_{i+1})$ are prime ideals
of $R$.  We claim that their intersection is the zero ideal; this will prove the theorem.  Since the rings
$R/\mf{p}_i$ all have the same dimension, there is no containment between the primes $\mf{p}_i$.  Since
$\mf{p}_1$ and $\mf{p}_2$ are both prime principal ideals, the lemma gives $\mf{p}_1 \cap \mf{p}_2=\mf{p}_1 \mf{p}_2$,
which is again a principal ideal.  A second application of the lemma gives $\mf{p}_1 \cap \mf{p}_2 \cap \mf{p}_3=
\mf{p}_1 \mf{p}_2 \mf{p}_3$.  However, since $(\alpha^2-1)m_{r+1}=0$, the product of the $\mf{p}_i$ is zero.  This
completes the proof.
\end{proof}

\begin{remark}
The point $c$ belongs to each of the three irreducible components of $\mc{C}$, and corresponds to the points $a$
and $b$ when these components are identified with $\mc{A}$ and $\mc{B}$.
\end{remark}

\section{Local deformation rings}
\label{s:local}

In this section we study certain ordinary deformation rings of local Galois groups.

\subsection{Set-up}

Let $F$ be a finite extension of $\Q_p$ of degree $d$ with absolute Galois group $G_F$.  Let $E$ be a finite extension
of $\Q_p$ with ring of integers $\mc{O}$, uniformizer $\pi$ and residue field $k$.  Let $V_0$ be a two dimensional
$k$-vector space
equipped with a continuous action of $G_F$ having cyclotomic determinant.  Let $\ms{C}_{\mc{O}}$ denote the category of
artinnian local $\mc{O}$-algebras with residue field $k$.  Define a functor\footnote{The appropriate 2-categorical
notions are to be understood throughout.} $D$ on $\ms{C}_{\mc{O}}$ by assigning to an algebra $A$ the groupoid of all
deformations of $V_0$ to $A$ with cyclotomic determinant.
That is, $D(A)$ is the category whose objects are free rank two $A$-modules $V$ equipped with a continuous action of
$G_F$ (for the discrete topology) having cyclotomic determinant together with an isomorphism $V \otimes_A k \to V_0$;
the morphisms in $D(A)$ are isomorphisms.

The category $D(A)$ is typically not discrete, and so $D$ will typically not be representable.  To circumvent this
annoyance,
we use framed deformations.  For $A$ as above, define $D^{\Box}(A)$ to be the category of pairs $(V, \{e_1, e_2\})$,
where $V$ is an object of $D(A)$ and $\{e_1, e_2\}$ is a basis of $V$ as an $A$-module.  Morphisms in $D^{\Box}(A)$
are required to respect the basis, and so there is at most one morphism between two objects.  The functor
$D^{\Box}(A)$ is pro-representable by a complete local noetherian $\mc{O}$-algebra $R^{\univ}$ together with a
free rank two $R^{\univ}$-module $V^{\univ}$.  The action of $G_F$ on $V^{\univ}$ is continuous when $V^{\univ}$
is given the $\mf{m}^{\univ}$-adic topology ($\mf{m}^{\univ}$ being the maximal ideal of $R^{\univ}$).

\subsection{Statement of problem}

For the purpose of this paper, we say that a continuous representation $\rho:G_F \to \GL_2(E')$, with $E'$ a finite
extension of $E$, is \emph{ordinary} if
\begin{displaymath}
\rho \vert_{I_F} \cong \mat{\chi}{\ast}{}{1}.
\end{displaymath}
Let $X$ be the locus in $\MaxSpec(R^{\univ}[1/\pi])$ consisting of ordinary representations.  Kisin's method,
discussed briefly below, shows
that $X$ is Zariski closed, and so there is a unique $\mc{O}$-flat reduced quotient $R$ of $R^{\univ}$ such that
$\MaxSpec(R[1/\pi])=X$.  This method also allows one to compute the components of $R[1/\pi]$ and show that each of
them is formally smooth over $E$.  However, as discussed in the introduction, the method does not provide much
information about $R$ itself.  Our goal is to understand this ring as best we can.

A natural way to proceed is to try to understand the deformation problem $R$ represents.  A good first guess is
that giving a map $R \to A$ is the same as giving a deformation $V$ of $V_0$ to $A$ which is in some sense
``ordinary.''  To this end, let us call such a $V$ \emph{ordinary} if there exists a rank one $A$-module summand $L$
of $V$ on which $G_F$ acts through the cyclotomic character and such that $G_F$ acts on $V/L$ trivially.  Let
$\ol{D}^{\ord}$ denote the functor assigning to $A$ the groupoid of ordinary deformations.

If $V_0$ is non-trivial then the functor $\ol{D}^{\ord, \Box}$ is represented by $R$.  This gives a moduli-theoretic
description of $R$, and allows $R$ to be studied relatively easily.  For instance, one can show
that the irreducible components of $\Spec(R)$ are formally smooth over $\mc{O}$ --- in fact, this comes out
of Kisin's analysis.

If $V_0$ is trivial then the functor $\ol{D}^{\ord, \Box}$ is not representable.  The problem is that, in this case,
a line $L$ in $V$ as above need not be unique.  The main idea of Kisin's method is to study the functor $D^{\ord}$
parameterizing pairs $(V, L)$, with $L$ as above.  The framed version of this functor is representable, but not by $R$;
rather, $\Spec(R)$ is the
scheme-theoretic image of the representing object in $\Spec(R^{\univ})$.  It is hard to deduce properites of $R$
from this description.  In the rest of \S \ref{s:local}, we will follow the plan outlined in the introduction to
obtain the equations cutting $R$ out from $R^{\univ}$, and then use this to analyze $R$.  (Actually, we work with a
slightly modified ring $\wt{R}$ defined below.)

{\it We assume for the rest of \S \ref{s:local} that $V_0$ is trivial.}

\subsection{Three sets of Galois representations}

Let $x$ be a point in $\MaxSpec(R^{\univ}[1/\pi])$ with residue field $E_x$.  Then $E_x$ is a finite extension of $E$,
and the action of $G_F$ on $V_x=V^{\univ} \otimes_{R^{\univ}} E_x$ is continuous when $V_x$ is given its
$p$-adic topology.

As defined above, we let $X \subset \MaxSpec(R^{\univ}[1/\pi])$ denote the set of points $x$ for which $V_x$ is
ordinary.  We now define three subsets of $X$.  We let $X_1$ be the subset consisting of points $x$ where $V_x$
is an extension of $E_x$ by $E_x(\chi)$ on the full Galois group.  Similarly, we let
$X_2$ denote the subset consisting of points $x$ where $V_x$ is an extension of $E_x(\eta)$ by
$E_x(\eta \chi)$ on the full Galois group, where $\eta$ is the unramified quadratic character.  Finally, we let
$X_3$ denote the subset of $X$ where the representation is crystalline.  We have the following basic result.

\begin{lemma}
The set $X$ is the union of the subsets $X_1$, $X_2$ and $X_3$.
\end{lemma}

\begin{proof}
Let $x \in X$.  Then $V_x \vert_{I_F}$ is an extension of $E_x$ by $E_x(\chi)$, and so $V_x$ is an extension of
$E_x(\psi^{-1})$ by $E_x(\psi \chi)$ for some unramified character $\psi$.  Thus $V_x$ defines an element of
$H^1(G_F, E_x(\chi \psi^2))$.
Now, we have isomorphisms
\begin{displaymath}
H^1(G_F, E_x(\chi \psi^2)) \cong H^1(I_F, E_x(\chi \psi^2))^{\Gal(F^{\un}/F)} = (((F^{\un})^{\times})^{\wedge} \otimes
E_x(\psi^2))^{\Gal(F^{\un}/F)}.
\end{displaymath}
The first isomorphism is restriction, the second comes from Kummer theory.  Here the $\wedge$ denotes $p$-adic
completion.  The valuation map on $(F^{\un})^{\times}$ defines a map
\begin{displaymath}
H^1(G_F, E_x(\chi \psi^2)) \to E_x(\psi^2)^{\Gal(F^{\un}/F)}.
\end{displaymath}
We thus see that either $\psi^2$ is trivial, in which case $x$ belongs to $X_1$ or $X_2$, or else the above
map is zero, in which case $x$ belongs to $X_3$.
\end{proof}

\subsection{The ring $R$}

The following result is due to Kisin.

\begin{proposition}[Kisin]
\label{local-1}
Let $\ast \in \{1, 2, 3\}$.  Then there exists a unique reduced $\mc{O}$-flat quotient $R_{\ast}$ of $R^{\univ}$
such that $\MaxSpec(R_{\ast}[1/\pi])$ is equal to $X_{\ast}$.  The ring $R_{\ast}$ is a domain which is
equidimensional of dimension $d+4$.  The ring $R_{\ast}[1/\pi]$ is regular.
\end{proposition}

We recall the relevant pieces of the argument.  For a complete proof, see \S 2.4 of Kisin's paper \cite{Kisin2}
or Conrad's unpublished notes \cite{Conrad}.  We just deal with the $\ast=3$ case; the other cases are similar.  Define
a functor $D_3$ on the category $\ms{C}_{\mc{O}}$ by assigning to an algebra $A$ the groupoid $D_3(A)$ of
pairs $(V, L)$ where $V \in D(A)$ and $L$ is a rank one $A$-module summand of $V$ on which $G_F$ acts through $\chi$
and such that the class in $H^1(G_F, L \otimes (V/L)^{\vee})$ determined by the extension
\begin{displaymath}
0 \to L \to V^{\univ}_A \to V^{\univ}_A/L \to 0
\end{displaymath}
belongs to $H^1_f(G_F, L \otimes (V^{\univ}_A/L)^{\vee})$ (see \cite[\S 2.4.1]{Kisin2} for the definition of $H^1_f$).
There is a natural
map $\Theta:D_3 \to D$ which forgets $L$.  This map is relatively representable and
projective; in fact, there is a closed immersion $D_3 \to \P^1_D$ lifting $\Theta$.  It follows that
$D_3^{\Box}$ is representable by a projective formal scheme $\wh{Z}$ over $\Spec(R^{\univ})$.
We let $Z$ be the algebraization of $\wh{Z}$, and still write $\Theta$ for the map $Z \to \Spec(R^{\univ})$.

The scheme $Z$ is formally smooth over $\mc{O}$.
The map $\Theta[1/\pi]$ is a closed immersion and induces a bijection between the closed points of $Z[1/\pi]$ and
the set $X_3 \subset \MaxSpec(R^{\univ}[1/\pi])$.  It follows that if we let $R_3$ be such that $\Spec(R_3)$ is the
scheme-theoretic image of $\Theta$, then $R_3$ is $\mc{O}$-flat and reduced and satisfies $\MaxSpec(R_3[1/\pi])=X_3$.
It is clear that $R_3$ is the unique quotient of $R^{\univ}$ with these properties.  Since $Z$ is formally smooth over
$\mc{O}$ and
$\Theta:Z[1/\pi] \to \Spec(R_3[1/\pi])$ is an isomorphism, it follows that $R_3[1/\pi]$ is formally smooth over
$E$, and thus regular.  One deduces that $R_3$ is a domain from the fact that the fiber of $\Theta$
over the closed point of $R^{\univ}$ is connected (it is $\P^1$ since $V_0$ is trivial); see \cite[Cor.~2.4.6]{Kisin}.
The dimension of $R_3$ can be calculated by looking at a tangent space.

From Proposition~\ref{local-1}, we immediately obtain the following theorem:

\begin{proposition}
There exists a unique reduced $\mc{O}$-flat quotient $R$ of $R^{\univ}$ with the property that $\MaxSpec(R[1/\pi])$
is equal to $X$.  The ring $R$ is equidimensional of dimension $d+4$ and has three minimal primes, namely the kernels
of the surjections $R \to R_{\ast}$.
\end{proposition}

\subsection{The ring $\wt{R}$}
\label{ss:rtilde}

Fix a Frobenius element $\phi$ of $G_F$.  We assume $\chi(\phi)=1$ for convenience.  Let $\wt{D}$ be the functor
on $\ms{C}_{\mc{O}}$ assigning to $A$ the set of
pairs $(V, \alpha)$, where $V$ belongs to $D(A)$ and $\alpha \in A$ is a root of the characteristic polynomial of
$\phi$ on $V$.  The framed version $\wt{D}^{\Box}=\wt{D} \times_D D^{\Box}$ is pro-representable by a complete
local noetherian $\mc{O}$-algebra $\wt{R}^{\univ}$.  The ring $\wt{R}^{\univ}$ is the quotient of $R^{\univ}[\alpha]$
by a monic degree two polynomial (the characteristic polynomial of $\phi$ on $V^{\univ}$).

We now define a map
\begin{displaymath}
\Phi:X \to \MaxSpec(\wt{R}^{\univ}[1/\pi]).
\end{displaymath}
Thus let $x$ be a point in $X$.  The space $V_x$ contains a unique line $L_x$ on which $I_F$ acts through the
cyclotomic character.  This line is stable by $G_F$ since $I_F$ is a normal subgroup.  Let $\alpha_x$ be the scalar
through which $\phi$ acts on $V_x/L_x$.  Then $\wt{x}=(V_x, \alpha_x)$
defines a point of $\MaxSpec(\wt{R}^{\univ})$.  We put $\Phi(x)=\wt{x}$.  Put $\wt{X}=\Phi(X)$ and
$\wt{X}_{\ast}=\Phi(X_{\ast})$.

The main result of this section is the following:

\begin{proposition}
Let $\ast \in \{1,2,3\}$.  Then there is a unique reduced $\mc{O}$-flat quotient $\wt{R}_{\ast}$ of $\wt{R}^{\univ}$
such that $\MaxSpec(\wt{R}_{\ast}[1/\pi])$ is equal to $\wt{X}_{\ast}$.  The ring $\wt{R}_{\ast}$ is a domain which is
equidimensional of dimension $d+4$.  The natural map $R_{\ast} \to \wt{R}_{\ast}$ is an isomorphism after
inverting $\pi$.
\end{proposition}

This proposition can be proved in the same manner as Proposition~\ref{local-1}.  Instead of doing this, we deduce it
from Proposition~\ref{local-1}, as this is a bit shorter.  We only treat the $\ast=3$ case, as the others are similar.

Let $Z$ be the scheme constructed in the previous section.  Let $\wt{\Theta}:Z \to \Spec(\wt{R}^{\univ})$ be the map
defined by taking $(V, L)$ to $(V, \alpha)$, where $\alpha$ is the scalar through which $\phi$ acts on $V/L$.  We
have a commutative diagrams
\begin{displaymath}
\xymatrix{
& Z \ar[ld]_{\wt{\Theta}} \ar[rd]^{\Theta} \\
\Spec(\wt{R}^{\univ}) \ar[rr] && \Spec(R^{\univ}) }
\end{displaymath}
where the bottom horizontal map is the natural one (forget $\alpha$), and
\begin{displaymath}
\xymatrix{
& Z[1/\pi]' \ar[ld]_{\wt{\Theta}} \ar[rd]^{\Theta} \\
\wt{X}_3 && X_3 \ar[ll]_{\Phi} }
\end{displaymath}
where the prime denotes the set of closed points.  Since $\Theta[1/p]$ is a closed immersion, it follows from the
first diagram that $\wt{\Theta}[1/p]$ is as well.  We thus see that $\wt{\Theta}$ is injective in the second diagram;
since we know that $\Theta$ and $\Phi$ and bijections, it follows that $\wt{\Theta}$ is as well.

Let $\wt{R}_3$ be such that $\Spec(\wt{R}_3)$ is the scheme-theoretic
image of $\wt{\Theta}$.  Since $Z$ is formally smooth over $\mc{O}$, the ring $\wt{R}_3$ is $\mc{O}$-flat
and reduced; furthermore, by the above comments, $\MaxSpec(\wt{R}_3[1/\pi])=\wt{X}_3$.  It is clear that $\wt{R}_3$ is
the
unique quotient of $\wt{R}^{\univ}$ with these properties.  Since $\Theta[1/\pi]:Z[1/\pi] \to \Spec(R_3[1/\pi])$ and
$\wt{\Theta}[1/\pi]:Z[1/\pi] \to \Spec(\wt{R}_3[1/\pi])$ are both isomorphisms, it follows from the first diagram
above that $R_3[1/\pi] \to \wt{R}_3[1/\pi]$ is an isomorphism.  We thus conclude that $\wt{R}_3$ is a domain and
equidimensional of dimension $d+4$ from the corresponding results for $R_3$.

\begin{remark}
The map $R_{\ast} \to \wt{R}_{\ast}$ is an isomorphism for $\ast \in \{1,2\}$.
\end{remark}

The above proposition immediately implies the following one.

\begin{proposition}
\label{local-2}
There is a unique reduced $\mc{O}$-flat quotient $\wt{R}$ of $\wt{R}^{\univ}$ such that $\MaxSpec(\wt{R}[1/\pi])$ is
equal to $\wt{X}$.  The ring $\wt{R}$ is equidimensional of dimension $d+4$ and has three minimal primes, the kernels
of the surjections $\wt{R} \to \wt{R}_{\ast}$.  The natural map $R \to \wt{R}$ is an isomorphism after inverting $\pi$.
\end{proposition}

\subsection{The ring $\wt{R}^{\dag}$}
\label{ss:rdag}

For $A \in \ms{C}_{\mc{O}}$, let $\wt{D}^{\dag}(A)$ denote the subset of $\wt{D}(A)$ consisting of those pairs
$(V, \alpha)$ such that the following conditions are satisfied:
\begin{itemize}
\item $\tr{g}=\chi(g)+1$ for all $g \in I_F$.
\item $(g-1)(g'-1)=(\chi(g)-1)(g'-1)$ for $g,g' \in I_F$.
\item $(g-1)(\phi-\alpha)=(\chi(g)-1)(\phi-\alpha)$ for $g \in I_F$.
\item $(\phi-\alpha)(g-1)=(\alpha^{-1}-\alpha)(g-1)$ for $g \in I_F$.
\end{itemize}
The functor $\wt{D}^{\dag, \Box}$ is clearly prorepresentable by a ring $\wt{R}^{\dag}$; to obtain $\wt{R}^{\dag}$,
simply form the quotient of $\wt{R}^{\univ}$ by the above equations.

\begin{lemma}
The natural map $\wt{R}^{\univ} \to \wt{R}$ factors through $\wt{R}^{\dag}$.
\end{lemma}

\begin{proof}
The map $\wt{\Theta}:Z \to \Spec(\wt{R}^{\univ})$ clearly factors through the closed immersion $\Spec(\wt{R}^{\dag})
\to \Spec(\wt{R}^{\univ})$, which proves the
lemma.
\end{proof}

\subsection{The main theorems}

We prove two main theorems.  The first is the following:

\begin{theorem}
\label{thm1}
The natural map $\wt{R}^{\dag} \to \wt{R}$ is an isomorphism.
\end{theorem}

This theorem gives a description of the points of $\wt{R}$.  Our second theorem is the following:

\begin{theorem}
\label{thm2}
Let $\ast \in \{1, 2, 3\}$.  The ring $\wt{R}_{\ast}$ is normal and Cohen--Macaulay but not Gorenstein.
\end{theorem}

The rest of this section is devoted to proving these two theorems.  We begin with some lemmas.

\begin{lemma}
\label{lem-a}
The natural map $\wt{R}^{\dag}[1/\pi]_{\red} \to \wt{R}[1/\pi]$ is an isomorphism.
\end{lemma}

\begin{proof}
Let $\wt{x}$ be a point in $\MaxSpec(\wt{R}^{\dag}[1/\pi])$ and let $(V, \alpha)$ be the corresponding representation
and eigenvalue of $\phi$.  Let $x$ be the image of $\wt{x}$ in $\MaxSpec(R^{\univ}[1/\pi])$, so that $V=V_x$.  The
equations of \S \ref{ss:rdag} hold on $V$.  The first of these, namely
$\tr(g)=\chi(g)+1$ for $g \in I_F$, shows that the semi-simplification of $V \vert_{I_F}$ is $\chi \oplus 1$.

Suppose that $V \vert_{I_F}$ is an extension of $\chi$ by 1, so that with respect to a suitable basis the action
of $I_F$ is given by
\begin{displaymath}
g \mapsto \mat{1}{f(g)}{}{\chi(g)}.
\end{displaymath}
The second equation of \S \ref{ss:rdag} shows that
\begin{displaymath}
(\chi(g)-1) f(g')=(\chi(g')-1) f(g)
\end{displaymath}
for all $g, g' \in I_F$, which shows that the cocycle $f$ is a coboundary (fix $g'$ with $\chi(g') \ne 1$).  Thus
the extension is split.

The previous paragraph shows that we can regard $V \vert_{I_F}$ as an extension of 1 by $\chi$.  Thus $x$ belongs
to $X$.  Let $\beta$ be the eigenvalue of $\phi$ on the inertial coinvariants of $V$.
Then the fourth equation in \S \ref{ss:rdag} shows that $(\beta^{-1}-\alpha) (g-1)=(\alpha^{-1}-\alpha) (g-1)$ holds on
$V_x$ for all $g \in I_F$.  This implies $\beta=\alpha$ (consider $g \in I_F$ with $\chi(g) \ne 1$), and so
$\wt{x}=\Phi(x)$.  This shows that $\wt{x}$ belongs to $\wt{X}$.

We have just shown that the inclusion $\wt{X}=\MaxSpec(\wt{R}[1/\pi]) \subset \MaxSpec(\wt{R}^{\dag}[1/\pi])$ induced
by the surjection $\wt{R}^{\dag} \to \wt{R}$ is in fact an equality.  The lemma follows.
\end{proof}

\begin{lemma}
\label{lem-b}
The ring $\wt{R}^{\dag}/\pi \wt{R}^{\dag}$ is isomorphic to the complete local ring of the scheme $\mc{C}_d$
at the point $c$ (see \S \ref{cspace} for the definition of $\mc{C}_d$ and $c$, and recall $d=[F:\Q_p]$).
\end{lemma}

\begin{proof}
Let $\wh{\mc{C}}$ denote the formal completion of $\mc{C}_d$ at the point $c$.  Let $\ms{C}_k$ denote the category of
complete local noetherian $k$-algebras with residue field $k$.  For $A \in
\ms{C}_k$, the set $\wh{\mc{C}}(A)$ consists of those elements of $\mc{C}(A)$ whose image in $\mc{C}(k)$ is the
point $c$.  We will show that the functors $\wt{D}^{\dag, \Box}$ and $\wh{\mc{C}}$ are isomorphic on the category
$\ms{C}_k$.  As $\wt{R}^{\dag, \Box}/\pi \wt{R}^{\dag}$ represents the former
functor, this will prove the lemma.

Let $A \in \ms{C}_k$.  We regard elements of $\wt{D}^{\Box}(A)$ as pairs
$(\rho, \alpha)$ where $\rho:G_F \to \GL_2(A)$ is a homomorphism reducing to the trivial homomorphism modulo the
maximal ideal of $A$ and $\alpha$ is an element of $A$ satisfying the characteristic polynomial of $\rho(\phi)$.
An element $(\rho, \alpha)$ of $\wt{D}^{\Box}(A)$ belongs to $\wt{D}^{\dag, \Box}(A)$ if and only if the following
equations hold:
\begin{itemize}
\item $\tr{\rho(g)}=2$ for all $g \in I_F$.
\item $(\rho(g)-1)(\rho(g')-1)=0$ for $g, g' \in I_F$.
\item $(\rho(g)-1)(\rho(\phi)-\alpha)=0$ for $g \in I_F$.
\item $(\rho(\phi)-\alpha)(\rho(g)-1)=(\alpha-\alpha^{-1})(\rho(g)-1)$ for $g \in I_F$.
\end{itemize}
These equations come from combining the defining equations of $\wt{D}^{\dag}$ with the assumption that $\chi$ reduces
to 1 modulo $p$.  Of course, we also have $\det{\rho(g)}=1$ for any such deformation.  These conditions imply that
$\rho(g)-1$ is strongly nilpotent for any $g \in I_F$, and thus $\rho(g)^p=1$ for any such $g$.  We thus see that if
$(\rho, \alpha)$ belongs to $\wt{D}^{\dag, \Box}(A)$ then $\rho \vert_{I_F}$ factors through the maximal abelian
quotient of $I_F$ of exponent $p$.  Of course, $\rho$ factors through the maximal $p$-power quotient of $G_F$ since its
reduction modulo the maximal ideal of $A$ is trivial.

Let $G$ be the maximal $p$-power quotient of $G_F$ in which inertia is abelian and of exponent $p$.  Let $U$ be the
inertia group in $G$.  We have a short exact sequence
\begin{displaymath}
0 \to U \to G \to \Z_p \to 0.
\end{displaymath}
The image of $\phi$ is a topological generator of $\Z_p$.  We give $U$ the structure of an $\F_p\lbb T \rbb$-module
by letting $T$ act by $\phi-1$.  As computed in Proposition~\ref{galstruct}, $U$ is isomorphic to
$\F_p \oplus \F_p \lbb T \rbb^{\oplus d}$.  Let $g_1, \ldots, g_d$ be an $\F_p \lbb T \rbb$-basis for the free
part of $U$ and let $g_{d+1}$ be a generator of the $T$-torsion of $U$.  Note that to give a continuous map from $G$ to
some discrete group $\Gamma$ is the same as giving elements $\ol{\phi}$ and $\ol{g}_1, \ldots, \ol{g}_{d+1}$ of
$\Gamma$ such that the $\ol{g}_i$ commute with each other, $\ol{g}_i^p=1$ for each $i$, $\ol{\phi}$ has finite order
and $\ol{\phi}$ and $\ol{g}_{d+1}$ commute.

For $A \in \ms{C}_k$, we define a map $\wt{D}^{\dag, \Box}(A) \to \wh{\mc{C}}(A)$ by taking $(\rho, \alpha)$ to
the tuple $(\phi, \alpha; m_1, \ldots, m_{d+1})$ where $\phi=\rho(\phi)$ (apologies for the bad notation) and
$m_i=\rho(g_i)-1$.  The defining equations for $\wt{D}^{\dag, \Box}$ given above show that this map is a
bijection.
\end{proof}

We now prove the first theorem.

\begin{proof}[Proof of Theorem~\ref{thm1}]
We follow the plan laid out in \S \ref{ss:out}.  We have already completed step (a) by guessing
the equations for $\wt{R}$ and defining the ring $\wt{R}^{\dag}$.  We now complete the process.
\begin{enumerate}
\setcounter{enumi}{1}
\item By Lemma~\ref{lem-a}, $\wt{R}^{\dag}[1/\pi]_{\red} \to \wt{R}[1/\pi]$ is an isomorphism.  In particular,
by Proposition~\ref{local-2},
$\wt{R}^{\dag}[1/\pi]$ is equidimensional of dimension $d+3$ and has three minimal primes.
\item By Lemma~\ref{lem-b} and Theorem~\ref{thm-c}, $\wt{R}^{\dag}/\pi \wt{R}^{\dag}$ is reduced, equidimensional
of dimension $d+3$ and has three minimal primes.
\item By Proposition~\ref{flat}, the ring $\wt{R}^{\dag}$ is flat over $\mc{O}$.  By Proposition~\ref{reduced} it is
reduced.
\item Since $\wt{R}^{\dag}$ is $\mc{O}$-flat and reduced and the map $\wt{R}^{\dag}[1/\pi]_{\red} \to \wt{R}[1/\pi]$
is an isomorphism, it follows that $\wt{R}^{\dag} \to \wt{R}$ is an isomorphism.
\end{enumerate}
This completes the proof.
\end{proof}

We now turn to the second theorem.

\begin{proof}[Proof of Theorem~\ref{thm2}]
As shown in the proof of Proposition~\ref{flat}, the ring $\wt{R}$ has three minimal primes, and these minimal primes
are naturally in correspondence with those of $\wt{R}/\pi \wt{R}$ and $\wt{R}[1/\pi]$.  It is clear that two of
these minimal primes are defined by the equations $\alpha=1$ and $\alpha=-1$.  The quotients by these minimal primes
are the rings $\wt{R}_1$ and $\wt{R}_2$.  We thus see that $\wt{R}_1/\pi \wt{R}_1$ and $\wt{R}_2/\pi \wt{R}_2$ are both
isomorphic to the complete local ring of $\mc{A}_{d+2}$ at $a$.  Finally, the third minimal prime gives $\wt{R}_3$.  It
is clear that $\wt{R}_3/\pi \wt{R}_3$ is isomorphic to the complete local ring of $\mc{B}_d$ at $b$, since this is the
only thing left over.

Appealing to Theorem~\ref{thm-a} and Theorem~\ref{thm-b}, we see that $\wt{R}_{\ast}/\pi \wt{R}_{\ast}$ is integral,
normal, Cohen--Macaulay and not Gorenstein.  The ring $R_{\ast}$ is a domain and $R_{\ast}[1/\pi]$ is normal (as it is
regular).  It follows that $R_{\ast}$ is normal (by Proposition~\ref{normal}), Cohen--Macaulay
(by \cite[Thm.~2.1.3a]{BrunsHerzog}) and not Gorenstein (by \cite[Prop.~3.1.19b]{BrunsHerzog}).
\end{proof}

\begin{remark}
We can analyze the rings $\wt{R}_1$ and $\wt{R}_2$ directly, without using the ring $\wt{R}$.  Indeed, we can
define $\wt{R}_1$ as the quotient of $\wt{R}$ be the equation $\alpha=1$, which realizes it directly as a quotient of
$R^{\univ}$.  We can then go through the above arguments, but specifically for $\wt{R}_1$.  However, we have not found
a way to analyze $\wt{R}_3$ directly:  we do not know the equations that cut it out from $\wt{R}$.  Note, however,
that we do know how to cut out $\wt{R}_3/\pi \wt{R}_3$ from $\wt{R}/\pi \wt{R}$:  it is defined by the equation
$\rho(g_{d+1})=1$, where $g_{d+1}$ is as in the proof of Lemma~\ref{lem-b}.  We have therefore studied $\wt{R}_3$
indirectly by studying the entire ring $\wt{R}$.
\end{remark}

\section{Global deformation rings}
\label{s:global}

In this section we give some global applications of the local results of the previous section.  These applications
are really just some superficial remarks, and reasonably well-known, so we do not bother going into many details.

\subsection{Torsion-freeness of deformation rings}

Let $E$, $\mc{O}$ and $k$ be as in the previous section.  Let $F$ be a totally real field, let $\Sigma$ be a finite
set of finite places of $F$, including all those above $p$, and let $\ol{\rho}:G_{F, \Sigma} \to \GL_2(k)$ be a
totally odd continuous representation of the absolute Galois group of $F$ unramified away from $\Sigma$.  We assume
that $\ol{\rho}$ is absolutely irreducible and has determinant $\chi$, the cyclotomic character.  We define several
deformation rings (all with fixed determinant $\chi$):
\begin{itemize}
\item For $v \in \Sigma$, let $R_v^{\Box, \univ}$ denote the universal framed deformation ring of
$\ol{\rho} \vert_{G_{F_v}}$.
\item For $v \in \Sigma$, choose a finite $R_v^{\Box, \univ}$-algebra $R_v^{\Box}$ which $\mc{O}$-flat and
equidimensional of dimension $[F_v:\Q_p]+4$ if $p \mid v$ or 4 if $p \nmid v$.
\item Let $R^{\Box, \univ}_{\loc}$ be the completed tensor product of the $R_v^{\Box, \univ}$ over $\mc{O}$ and let
$R^{\Box}_{\loc}$ be the completed tensor product of the $R_v^{\Box}$ over $\mc{O}$.
\item Let $R^{\Box, \univ}$ be the universal deformation ring of $\ol{\rho}$ with framings at each $v \in \Sigma$ and
let $R^{\Box}$ be the completed tensor product of $R^{\Box, \univ}$ with $R^{\Box}_{\loc}$ over
$R^{\Box, \univ}_{\loc}$.
\item Let $R^{\univ}$ be the universal (unframed) deformation ring of $\ol{\rho}$ and let $R$ be the descent of
$R^{\Box}$ from $R^{\Box, \univ}$ to $R^{\univ}$; it is a finite $R^{\univ}$-algebra.
\end{itemize}
We then have the following result, taken from the discussion in \cite{KhareWintenberger} before
Corollary~4.7.

\begin{proposition}
Assume $R$ is finite over $\mc{O}$ and each $R_v^{\Box}$ is Cohen--Macaulay.  Then $R$ is flat over $\mc{O}$ and
Cohen--Macaulay.  Furthermore, $R$ is Gorenstein if and only if each $R_v^{\Box}$ is.
\end{proposition}

\begin{proof}
By \cite[Prop.~4.1.5]{Kisin3}, we have a presentation
\begin{displaymath}
R^{\Box, \univ}=R^{\Box, \univ}_{\loc} \lbb x_1, \ldots, x_{r+n-1} \rbb/(f_1, \ldots, f_{r+s}),
\end{displaymath}
where $n=\# \Sigma$, $s=[F:\Q]$ and $r$ is a non-negative integer.  Tensoring over $R^{\Box, \univ}_{\loc}$
with $R^{\Box}_{\loc}$ gives a presentation
\begin{displaymath}
R^{\Box}=R^{\Box}_{\loc} \lbb x_1, \ldots, x_{r+n-1} \rbb/(f_1, \ldots, f_{r+s}).
\end{displaymath}
This shows that $R^{\Box}$ has dimension at least $4n$.  Since $R^{\Box}$ is a power series ring over $R$ in
$4n-1$ variables and $R$ is finite over $\mc{O}$, we see that $R$ has dimension 1 and $R^{\Box}$ has dimension $4n$.
Furthermore, writing $R^{\Box}=R\lbb T_1, \ldots, T_{4n-1} \rbb$, we see that $T_1, \ldots, T_{4n-1}, f_1, \ldots,
f_{r+s}, p$ is a system of parameters for $R^{\Box}_{\loc} \lbb x_1, \ldots, x_{r+n-1} \rbb$.  Since each
$R^{\Box}_v$ is Cohen--Macaulay, so too is $R^{\Box}_{\loc} \lbb x_1, \ldots, x_{r+n-1} \rbb$, and it follows that
$T_1, \ldots, T_{4n-1}, f_1, \ldots, f_{r+s}, p$ is a regular sequence.  This shows that $R^{\Box}$ is $\mc{O}$-flat
and Cohen--Macaulay, and furthermore that $R^{\Box}$ is Gorenstein if and only if each $R^{\Box}_v$ is.  Finally, note
that $R^{\Box}$ is a power series ring over $R$, and so all these properties can be transferred to $R$.
\end{proof}

\begin{remark}
Finiteness of $R$ is known in many cases.  When $\ol{\rho}$ is modular, one can often obtain finiteness of $R$ using
the Taylor--Wiles argument as modified by Kisin.  Even without modularity it is often still possible to obtain
finiteness by using potential modularity.
\end{remark}

The following proposition is a very special case, showing how the above proposition can be combined with the
main results of this paper.

\begin{proposition}
Suppose $\Sigma$ consists exactly of the primes over $p$ and that for each $v \in \Sigma$ the local representation
$\ol{\rho} \vert_{G_{F_v}}$ is trivial.  For $v \in \Sigma$, let $R^{\Box}_v$ be one of the rings $\wt{R}_{\ast}$
constructed in \S \ref{ss:rtilde}.  Then, assuming $R$ is finite over $\mc{O}$, it is $\mc{O}$-flat, Cohen--Macaulay
and not Gorenstein.
\end{proposition}

\subsection{An $R=\T$ theorem}

Assume that $\ol{\rho}$ is modular.  Then one can define a Hecke algebra $\T$ and a surjection $R \to \T$.  The
original method of Taylor and Wiles shows that this map is an isomorphism in certain situations.  Kisin's modification
of the method applies in greater generality, but only shows that $R[1/p] \to \T[1/p]$ is an isomorphism.

We simply remark here that if one knows that $R$ is $\mc{O}$-flat, then knowing that $R[1/p] \to \T[1/p]$ is an
isomorphism implies that the map $R \to \T$ is an isomorphism.  Thus when $\ol{\rho}$ is trivial at the places
above $p$ and one uses the $\wt{R}_{\ast}$ local deformation conditions, one can expect to obtain an $R=\T$ theorem.


\begin{thebibliography}{[KW]}

\bibitem[BH]{BrunsHerzog}
W.~Bruns and J.~Herzog, \emph{Cohen--Macaulay rings.}
Cambridge Studies in Advanced Mathematics, 39. Cambridge University Press, Cambridge, 1993.

\bibitem[BM]{BarcanescuManolache}
\c{S}.~B\u{a}rc\u{a}nescu and N.~Manolache,
\emph{Betti numbers of Segre-Veronese singularities.}
Rev.\ Roumaine Math.\ Pures Appl.\ {\bf 26} (1981), no.~4, 549--565.

\bibitem[C]{Conrad}
B.~Conrad, \emph{Structure of ordinary-crystalline deformation ring for $\ell=p$.}  Unpublished notes, available at
\url{http://math.stanford.edu/~conrad/modseminar/pdf/L21.pdf}

\bibitem[H]{Hartshorne}
R.~Hartshorne, \emph{Residues and duality.}
Lecture Notes in Mathematics, No.~20 Springer-Verlag, 1966.

\bibitem[K]{Kisin}
M.~Kisin, \emph{Moduli of finite flat group schemes, and modularity.}
Ann.\ of Math.\ (2) {\bf 170} (2009), no.~3, 1085--1180.

\bibitem[K2]{Kisin2}
M.~Kisin, \emph{Modularity of 2-adic Barsotti-Tate representations.}
Invent.\ Math.\ {\bf 178} (2009), no.~3, 587--634.

\bibitem[K3]{Kisin3}
M.~Kisin, \emph{Modularity of 2-dimensional Galois representations.}
Current Developements in Mathematics (2005), 191--230.

\bibitem[KW]{KhareWintenberger}
C.~Khare and J.-P.~Wintenberger, \emph{Serre's modularity conjecture. II.}
Invent.\ Math.\ {\bf 178} (2009), no.~3, 505--586.

\bibitem[W]{Weyman}
J.~Weyman,
\emph{Cohomology of vector bundles and syzygies.}
Cambridge Tracts in Mathematics, 149. Cambridge University Press, Cambridge, 2003.

\end{thebibliography}
\end{document}